\documentclass[11pt]{article}

\usepackage[top=1in, bottom=1in, left=1.25in, right=1.25in, marginparwidth=1in, marginparsep=0.1in]{geometry}

\usepackage{amsmath,amssymb,amsthm,xcolor,enumerate,graphicx}
\usepackage{esint} 
\usepackage[unicode,breaklinks=true,colorlinks=true,citecolor = {magenta}]{hyperref}
\usepackage{multirow,diagbox}
\usepackage{float}
\restylefloat{table}

\numberwithin{equation}{section}
\newtheorem{theorem}{Theorem}[section]

\newtheorem{lemma}[theorem]{Lemma}
\newtheorem{proposition}[theorem]{Proposition}
\newtheorem{conjecture}[theorem]{Conjecture}

\theoremstyle{remark}
\newtheorem{remark}[theorem]{Remark}
\theoremstyle{definition}

\newcommand{\bke}[1]{\left ( #1 \right )}

\newcommand{\bket}[1]{\left \{ #1 \right \}}
\newcommand{\norm}[1]{ \| #1  \|}

\newcommand\be{\beta}

\newcommand\de{\delta}

\renewcommand\th{\theta}

\newcommand\la{\lambda}
\newcommand\si{\sigma}

\newcommand\ph{\varphi} %

\newcommand\om{\omega}

\newcommand\De{\Delta}

\newcommand\Om{\Omega}

\newcommand{\R}{\mathbb{R}}
\newcommand{\CC}{\mathbb{C}}
\newcommand{\Z}{\mathbb{Z}}
\newcommand{\N}{\mathbb{N}}

\renewcommand{\Re} {\mathop{\mathrm{Re}}}
\renewcommand{\Im} {\mathop{\mathrm{Im}}}

\renewcommand{\div}{\mathop{\rm div}}
\newcommand{\curl} {\mathop{\rm curl}}

\newcommand{\pd}{\partial}
\newcommand{\nb}{\nabla}
\newcommand{\td}{\tilde}

\renewcommand{\bar}[1]{\overline{#1}}
\newcommand{\lec}{{\ \lesssim \ }}

\newcommand{\sB}{\mathfrak{B}}

\newcommand{\cF}{\mathcal{F}}

\renewcommand{\[}{\begin{equation*}}
\renewcommand{\]}{\end{equation*}}
\newcommand{\bb}{\begin{equation}}
\newcommand{\ee}{\end{equation}}
\newcommand{\EQ}[1]{\begin{equation*}\begin{split} #1 \end{split}\end{equation*}}
\newcommand{\EQN}[1]{\begin{equation}\begin{split} #1 \end{split}\end{equation}}

\newcommand{\donothing}[1]{}
\newcommand{\sL}{\mathfrak{L}}

\newcommand{\bq}{\begin{eqnarray*}}
\newcommand{\eq}{\end{eqnarray*}}
\newcommand{\ba}{\begin{align}}
\newcommand{\ea}{\end{align}}
\newcommand{\pa}{\partial}
\newcommand{\na}{\nabla}
\newcommand{\tht}{\theta}

\newcommand{\del}{\delta}
\renewcommand{\hat}{\widehat}
\newcommand{\sM}{\mathcal{M}}

\newcommand{\loc}{\text{loc}}

\begin{document}

\title{On bifurcation of self-similar solutions of the stationary Navier-Stokes equations}
\author{Hyunju Kwon \href{https://orcid.org/0000-0003-4093-3991}{\includegraphics[width=8pt]{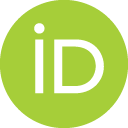}} \and 
Tai-Peng Tsai \href{https://orcid.org/0000-0002-9008-1136}{\includegraphics[width=8pt]{ORCID-iD_icon-128x128.png}}}
\date{Dedicated to Vladimir \v Sver\'ak on the occasion of his 60th birthday}
\maketitle

\begin{abstract} Landau solutions are special solutions to the stationary incompressible Navier-Stokes equations in the three dimensional space excluding the origin. They are self-similar and axisymmetric with no swirl. In fact, any self-similar smooth solution must be a Landau solution. In the effort of extending this result to the one for the solution class with the pointwise scale-invariant bound $|u(x)|\leq C_0|x|^{-1}$ for some $C_0>0$, we consider axisymmetric discretely self-similar solutions, and investigate the existence of such solution curve emanating from some Landau solution. We prove that the inclusion of the swirl component does not enhance the bifurcation and present numerical evidence of no bifurcation.



\end{abstract}

\section{Introduction}
This paper is concerned with solutions of the stationary incompressible Navier-Stokes equations
\begin{equation}
\label{SNS1} \begin{cases}
-\De u + (u \cdot \nb) u + \nb p =0, \\
\div u = 0
\end{cases}
\end{equation}
in $\Om= \R^3 \setminus \{ 0\}$ that satisfies%
\begin{equation}
\label{SNS2} 
|u(x)| \le \frac {C_0}{|x|}, \quad (x \in \Om),
\end{equation}
for some $C_0>0$. Here $u : \Om \to \R^3$ is the velocity field and $p:\Om \to \R$ is the pressure.
The usual regularity theory implies that $u$ is smooth. In fact, by \cite{Sverak-Tsai}, 
\begin{equation}
\label{SNS3}
|\nb^k u(x)|  \le \frac {C_k}{|x|^{k+1}}, \quad (x \in \Om),
\end{equation}
for some $C_k= C_k(C_0)$, for all $k \in \N$.
The system \eqref{SNS1} enjoys the \emph{scaling property}: If $(u,p)$ is a solution pair in $\Om$, then for any $\la>0$,
\[
u^\la(x)=\la u(\la x), \quad p^\la (x) = \la^2 p(\la x)
\]
is also a solution pair in $\Om$. A solution pair $(u,p)$ is called \emph{self-similar} if $(u^\la,p^\la)=(u,p)$ for all $\la>0$. 
In this case, $u$ and $p$ are homogeneous of degree $-1$ and $-2$, respectively,
\[
u(x) = \frac 1{|x|} u\bke{ \frac x{|x|}},\quad p(x) = \frac 1{|x|^2} p\bke{ \frac x{|x|}}.
\] 
A solution pair $(u,p)$ is called \emph{discretely self-similar (DSS)}  if $(u^\la,p^\la)=(u,p)$ for one $\la>1$. 
In this case, $u$ may not be minus one homogeneous, but if $u \in H^1_\loc(\Om)$ then it enjoys the estimates \eqref{SNS2} and  \eqref{SNS3}. It is determined by its value in the annulus $B_\la \setminus B_1$, where $B_r = \bket{x \in \R^3: \  |x| <r }$.

A special family of solutions of \eqref{SNS1}-\eqref{SNS2} is the \emph{Landau solutions} or \emph{Slezkin-Landau solutions}, computed by Slezkin \cite{Slezkin} in 1934 (see \cite{Galaktionov} for English translation),
and by Landau in 1944 \cite{Landau}. Landau's computation can be found in standard textbooks \cite[\S23]{Landau-Lifshitz} and \cite[\S4.6]{Bat}. 
The solutions were also independently found by Squire \cite{Squire} in 1951, and more recently revisited in Tian and Xin \cite{Tian-Xin} and Cannone and Karch \cite{MR2034160}.
These solutions are self-similar and axisymmetric with no swirl. 
In spherical coordinates $(\rho, \theta, \phi)$ with 
\begin{equation}
\label{spherical-coordinates}
(x_1,x_2,x_3) = (\rho \sin \phi \cos \th, \rho \sin \phi \sin \th,  \rho \cos \phi),
\end{equation}
and basis vectors
\[
e_\rho = \frac x\rho, \quad e_\theta = (-\sin \th, \cos \th, 0), \quad e_\phi = e_\th \times e_\rho,
\]
a function $f$ is called \emph{axisymmetric} if $f=f(\rho, \phi)$ is independent of $\th$, and
a vector field $u$ is \emph{axisymmetric} if it is of the form
\[
u = u_\rho(\rho, \phi) e_\rho + u_\th(\rho, \phi) e_\th + u_\phi(\rho, \phi) e_\phi
\]
with components $u_\rho$, $u_\th$ and $u_\phi$ independent of $\th$. It has \emph{no swirl} if the swirl component $u_\th$ is zero. Both classes of axisymmetric flows and  axisymmetric flows with no swirl are invariant under \eqref{SNS1}: If $(u,p)$ is axisymmetric, then the left side of \eqref{SNS1} is also axisymmetric. Similarly if $u$ has no swirl. Thus these two classes are preserved under time evolution if we add $\pd_t u$ to the left side of \eqref{SNS1}$_1$.

The Landau solutions, denoted by $U^a$ with parameter $a>1$, are
\begin{equation}
\label{Landau-sol}
{ U^a = \frac 2{\rho}\bke{\frac {a^2-1}{(a-\cos \phi)^2} -1} e_\rho
  + 0 e_\th- \frac {2\sin \phi} {\rho(a-\cos \phi)} e_\phi,}
 \quad P^a =\frac {4(a\cos \phi -1)}{\rho^2(a-\cos \phi)^2}. 
\end{equation}
 It can also be  written as
\begin{equation}
\label{Landau-Psi}
U^a = \curl (\Psi^a e_\th), \quad \Psi^a =  \frac {2\sin \phi }{a-\cos \phi}.
\end{equation}
The Landau solution $U^a$ satisfies the (inhomogeneous) Navier-Stokes equations with delta force
at the origin,
\begin{equation}
\label{Landau-eq} -\De u + (u \cdot \nb) u + \nb p ={\vec\be \de_0}, \quad
\div u = 0
\end{equation}
in $\R^3$, where $\de_0$ is the Dirac delta function at the origin, $\vec\be= \beta e_3$ and
\[
\beta= \beta_0(a) = 16 \pi \bke{a +
	\frac 12 a^2 \log \frac {a-1}{a+1} + \frac {4a}{3(a^2-1)}},
\]
see \cite[Lemma 8.2]{nslec}.
The function $a\in (1,\infty] \mapsto \beta_0 \in [0,\infty)$ is strictly decreasing, one to one and onto.
Note that $\beta_0(a)$ and the bound $C_0$ in \eqref{SNS2} for $U^a$ go to infinity as $a \to 1_+$.
In the literature,
instead of $a$, one sometimes uses $\vec\beta \in \R^3$ as the parameter and denotes the Landau solution 
as $U^{\vec\be}$ or $U^\be$. The basis $\{e_1,e_2,e_3\}$ is then changed accordingly so that $e_3$ is in the direction of $\vec\beta$.

Landau solutions appear as the asymptotic leading terms of solutions of \eqref{SNS1} in exterior domains in $\R^3$:
Nazarov and Pileckas \cite{NP00} derived asymptotic expansion for solutions of \eqref{SNS1} satisfying the bounds \eqref{SNS2}-\eqref{SNS3}
under smallness conditions, but the leading term was less explicit. Korolev and \v Sver\'ak \cite{Korolev-Sverak} showed that the leading term of a small solution must be a Landau solution. 
This result was
extended to small time-periodic solutions by  Kang, Miura and Tsai \cite{KMT12}, identifying that the leading spatial term is a fixed time-independent Landau solution.
Decaster and Iftimie \cite{MR3610933} extends the asymptotic results of \cite{NP00,Korolev-Sverak} to stationary solutions in an exterior domain with  minus-three homogeneous force fields. The existence of solutions with minus-three homogeneous \emph{axisymmetric} force fields in the whole space $\R^3$ is addressed by Shi \cite{Shi}. 
The Landau solutions are also useful to describe the local behavior near a 
singularity. Indeed, it was proved by Miura and Tsai \cite{MT12} that the
leading term of point singularity like $|x|^{-1}$ at $x=0$ of the Navier-Stokes flow is
also given by a Landau solution provided it is small enough.
See Hishida \cite{Hishida} for a survey including stationary Navier–Stokes flows around a rotating body.

The papers \cite{Slezkin, Landau, Squire, Tian-Xin, MR2034160} study self-similar solutions of \eqref{SNS1} in the axisymmetric class. In the 
axisymmetric class, \eqref{SNS1} is reduced to an ODE system, and can be analyzed by ODE techniques. This has been extended by Li, Li and Yan \cite{MR3744383,MR3770045,Li-Li-Yan3} to axisymmetric self-similar solutions with point singularities at the north and south poles on $\mathbb{S}^2$. 

Without assuming axisymmetry,
it has been shown by \v Sver\'ak \cite{Sverak2011} that,
if a solution of \eqref{SNS1} in $\Om$ is self-similar, then it must be a Landau solution. His analysis reduces \eqref{SNS1} to a PDE system on the unit sphere $\mathbb{S}^2$ using the self-similarity assumption.

To take one step further, we would like to ask: Do we have any strictly DSS solution of \eqref{SNS1}?
This is motivated by the following.

\begin{conjecture} 
A nonzero solution of \eqref{SNS1} in $\R^3 \setminus \{0\}$ satisfying the bound \eqref{SNS2} must be a Landau solution.
\end{conjecture} 

In this conjecture, no assumption is made on self-similarity or axisymmetry. The conjecture is first formulated in \cite{Sverak2011}, and is
known to be true if the constant $C_0$ in  \eqref{SNS2} is sufficiently small by Korolev-\v Sver\'ak \cite{Korolev-Sverak} and Miura-Tsai \cite{MT12} via different proofs.
For ``large'' solutions, 
the paper \cite{Sverak2011} by \v Sver\'ak
excludes a counterexample in the class of self-similar flows. One is naturally led to look for  a counterexample in the class of DSS flows with axisymmetry.

Such solutions may come in two kinds. Solutions of the first kind occur as bifurcations of Landau solutions. Solutions of the second kind occur in isolated island(s) in the class of DSS axisymmetric flows and stay away from the Landau solutions. This paper investigates the first kind only.

There is a rich literature on bifurcations of fluid PDEs. The most relevant to us is that of the Couette-Taylor flows by
 Velte \cite{Velte66}, presented in Temam \cite[II.4]{Temam} (see also \cite{nslec}). 

Under the ansatz of discrete self-similarity with DSS factor $\la>1$, a solution $u$ is determined by its value in the region
\[
\bket{x \in \R^3: \ 1 \le |x| \le \la },
\]
with the boundary condition
\[
u( x) = \la u(\la x),\quad ( |x|=1).
\]
Assuming axisymmetry in spherical coordinates $(\rho, \theta, \phi)$, the components $u_\rho$, $u_\th$,  and $u_\phi$ do not depend on $\th$, and there is an axisymmetric stream function $\psi$ so that $u_\rho e_\rho  + u_\phi e_\phi=\curl (\psi e_\th)$. Introduce the new variable
\[
\tau = \ln \rho, \quad 0 \le \tau \le \ln \la,
\]
and let
\[
\underline\psi(\tau,\phi):=\psi(\rho,\phi),\quad 
\chi(\tau,\phi):=\rho u_{\tht}(\rho,\phi) .
\]
They satisfy periodic boundary condition in $\tau$ and Dirichlet/Neumann boundary conditions in $\phi$. The solution pair corresponding to a Landau solution $U^a$ is $(\Psi^a,0)$.
We study the nonlinear equations for perturbations of $(\Psi^a,0)$. The null space of its linearized operator always contains $(\frac{\pd }{\pd a} \Psi_a,0)$. Looking for saddle-node type bifurcations, we try to find an additional element of the null space by varying the parameters $a$ and $\la$. Recall that $\beta_0(a) \to \infty $ as $ a \to 1_+$. Bifurcation does not occur for $a$ sufficiently large by \cite{Korolev-Sverak,MT12}, and is more likely to happen for $a$ close to $1$. We will also consider the general eigenvalue problems. Because the functions are periodic in the radial variable $\tau$, we can consider the restrictions of these eigenvalue problems to Fourier subspaces of $\tau$.

This paper contains both analytical and numerical results. Analytically, we show that the inclusion of the swirl component $u_\th$
does not enhance the bifurcation. In other words, if a bifurcation does occur, it can happen with $u_\th=0$. Numerically, we show evidences that bifurcation does not occur if $a\ge 1.01$.  Therefore, our results suggest that there is no axisymmetric discretely self-similar solution curve emanated from a
Landau solution, when $a\ge 1.01$.

The rest of the paper is organized as follows.
We first comment on  time dependent settings in Subsection \ref{sec1.1}.
 In Section \ref{sec2}, we deduce the nonlinear equations for DSS axisymmetric solutions in terms of the swirl component and the stream function.
In Section \ref{sec3}, we identify the linearization of the above system and rewrite them in similarity variables. By restricting to Fourier subspaces of the radial variable, we change its eigenvalue problem to eigenvalue problems for ordinary differential operators. In Section \ref{sec4}, we analyze the eigenvalue problems for the swirl operator, and prove Theorem \ref{eig.sM.thm}, which implies that the swirl component does not enhance bifurcation. In Section \ref{sec5}, weanalyze the eigenvalue problems for the stream operator, and provide numerical evidence for no bifurcation of Landau solutions in the class of DSS axisymmetric flows with no swirl when $a\geq 1.01$.

\subsection{Comments on time dependent settings}\label{sec1.1}

To place Landau solutions in a time dependent setting, we may consider
\begin{equation}\label{0521}
\pd_t u - \De u + u \cdot \nb u + \nb p = \be(t) e_3\de_0,
\quad \div u=0.
\end{equation}
A stationary solution is $u(x,t)=U^a(x)$ with $\be(t)\equiv \be_0(a)$. The $L^2$-stability of such a fixed Landau solution is 
studied by Cannone and Karch \cite{MR2034160}. Even for constant $\beta$, there are more possibilities of bifurcation for time dependent $u$. When $\be(t)$ is time dependent, $\pd_a \Psi$ plays a more significant role. In this case, in the equation of the difference $\tilde u(t) = u(t) - U^{a(t)}$, the term $- a'(t) \pd_a U^{a}$ appears on the right side. 

Eq.~\eqref{0521} may not seem physical. However, we may study solutions of a coupled system of Navier-Stokes equations with another equation (for heat, magnetic field, etc), and use \eqref{0521} as a model or a truncated version.

Let us consider time dependent solutions of \eqref{0521} with 
constant $\be(t)\equiv \be_0(a)$ and examine what kind of equations we will get under similar change of variables as in Section \ref{sec3}.
We keep $\pa_t u$ in our derivation, starting from \eqref{NSaxisym1}-\eqref{NSaxisym3}. It creates $\pa_t \hat{A}\psi$ in the equation for the stream function and $\pa_t u_\th$ in the equation for the swirl component. (See Section \ref{sec3}  for the definitions of $\hat{A}$.) Proceeding the linearization around a Landau solution $(U^a, P^a)$ as  in Section \ref{sec3}, we obtain in \eqref{operator.form} and \eqref{eig.prob} the time-dependent stream and swirl operators, $\hat{\sL^t}$ and $\hat{\sM^t}$, defined by
\begin{align*}
\hat{\sL^t} = \hat{\sL} + \pa_t \hat{A}, 
\quad
\hat{\sM^t} = \hat{\sM} + \pa_t.
\end{align*}
(See Section \ref{sec3}  for the definitions of $\hat{\sL}$ and $\hat{\sM}$.) In addition to the similarity radial variable $\tau=\ln \rho$ to be defined in \eqref{ch.var} and suitable for DSS perturbations, we introduce a similarity time variable $s$,
\begin{align*}
s= \rho^{-2}t, \quad \rho^2\pa_t = \pa_s.
\end{align*}
The factor $\rho^{-2}$ is needed to fit the scaling.
Since $s=s(t,\rho)$, the derivative $\rho\pa_\rho$ in the time-dependent case contains an extra term with derivative in $s$: $\rho\pa_\rho=\pd_\tau -2s\pa_s$. If we define $\xi(t,\rho,\phi) = \zeta(s,\tau,\phi)$ and $\sL^t \zeta = \rho^4 \hat{\sL^t} \xi$, then $\sL^t$ contains up to third derivative in $s$. 
One can transfer the eigenvalue problem of $\sL^t$  into a first order system in $s$ and analyze a new eigenvalue problem.
Of this new problem, $s$-independent solutions are exactly those to be considered in Sections \ref{sec4} and \ref{sec5}. We may also consider $s$-periodic solutions of this new system, corresponding to a Hopf bifurcation. However,  periodicity in $s$ does not mean periodicity in $t$ or DSS in $t$. Thus Hopf bifurcation is possible, but requires clarification of its meaning.

Motivated by this, we will also consider purely imaginary eigenvalues in Sections \ref{sec4} and \ref{sec5}, but those results have no implication on Hopf bifurcation.

\section{Equations for DSS axisymmetric flows}\label{sec2}

In this section, we describe the stationary Navier-Stokes equations for DSS, axisymmetric flows. Indeed, an axisymmetric flow $u$ can be written in terms of a stream function $\psi$ and a swirl velocity $u_\th$, 
\[
u = \curl(\psi e_\th) + u_\th e_\th.
\]
Also, the axisymmetry reduces our domain %
to a two dimensional space. %
As a consequence, stationary Navier-Stokes equations with axisymmetry reduce to the equation for $\psi$ and $u_\th$ and we impose natural boundary conditions for smooth solutions on the restricted domain.

\subsection{The equations for DSS axisymmetric steady flows}

In this subsection, we introduce the Navier-Stokes equations for axisymmetric steady flows in the spherical coordinates. We first recall the time-dependent Navier-Stokes equations in the spherical coordinates $(\rho,\th,\phi)$ without axisymmetry assumption (see \cite[Appendix 2]{Bat})
\vspace{-3mm}
\begin{multline*}
\partial_t{u_{\rho}}+ ({u}\cdot\nabla) u_{\rho}-\frac{u^2_{\phi}}{\rho}-\frac{u^2_{\theta}}{\rho}
=-\frac{1}{\rho_0}\partial_\rho p
\\
+\nu\bke{\Delta
u_{\rho}-\frac{2u_{\rho}}{\rho^2}-\frac{2}{\rho^2\sin\phi}\partial_{\phi}(\sin\phi\,
u_{\phi})-\frac{2}{\rho^2\sin\phi}\partial_{\theta} u_{\theta}},
\end{multline*}
\vspace{-5mm}
\begin{multline*}
\partial_t{u_{\phi}}+ ({u}\cdot\nabla) u_{\phi}+\frac{u_\rho u_{\phi}}{\rho}-\frac{u^2_{\theta}\cot\phi}{\rho}
=-\frac{1}{\rho_0\rho}\partial_{\phi} p
\\
+\nu\bke{\Delta
u_{\phi}+\frac{2}{\rho^2}\partial_{\phi}
u_{\rho}-\frac{u_{\phi}}{\rho^2\sin^2\phi}
-\frac{2\cos\phi}{\rho^2\sin^2\phi}\partial_{\theta}u_{\theta}},
\end{multline*}
\vspace{-5mm}
\begin{multline*}
\partial_t{u_{\theta}}+ ({u}\cdot\nabla)
u_{\theta}+\frac{u_{\theta}u_\rho}{\rho}+\frac{u_{\phi}u_{\theta}\cot\phi}{\rho}
=-\frac{1}{\rho_0\rho\sin\phi}\partial_{\theta} p
\\
+\nu\bke{\Delta
u_{\theta}+\frac{2}{\rho^2\sin\phi}\partial_{\theta}
u_{\rho}+\frac{2\cos\phi}{\rho^2\sin^2\phi}\partial_{\theta}u_{\phi}
-\frac{u_{\theta}}{\rho^2\sin^2\phi}},
\end{multline*}
$$
\nabla\cdot {u}
=\frac{1}{\rho^2}\partial_{\rho}(\rho^2u_{\rho})+\frac{1}{\rho\sin\phi}\partial_{\phi}(\sin\phi\,
u_{\phi})+\frac{1}{\rho\sin\phi}\partial_{\theta}u_{\theta} =0.
$$
Here, $\nu>0$ is the viscosity constant and $\rho_0>0$ is the constant density. We set $\nu=\rho_0=1$ below.

For axisymmetric steady flows, the components $u_\rho$, $u_\phi$ and $u_\th$ do not depend on $t$ or $\th$, and
the above system becomes 
\begin{align}
\label{NSaxisym1}
&({b}\cdot \nabla) u_{\rho}-\frac{u^2_{\phi}}{\rho}-\frac{u^2_{\theta}}{\rho}
=-\partial_\rho p
+\Delta
u_{\rho}-\frac{2u_{\rho}}{\rho^2}-\frac{2}{\rho^2\sin\phi}\partial_{\phi}(\sin\phi\,u_{\phi}),
\\
&({b}\cdot \nabla) u_{\phi}+\frac{u_\rho u_{\phi}}{\rho}-\frac{u^2_{\theta}\cot\phi}{\rho}
=-\frac{1}{\rho}\partial_{\phi} p 
+ \Delta u_{\phi}+\frac{2}{\rho^2}\partial_{\phi}
u_{\rho}-\frac{u_{\phi}}{\rho^2\sin^2\phi},
\\
\label{NSaxisym3}
&({b}\cdot\nabla)
u_{\theta}+\frac{u_{\theta}u_\rho}{\rho}+\frac{u_{\phi}u_{\theta}\cot\phi}{\rho}
=
\Delta u_{\theta}
-\frac{u_{\theta}}{\rho^2\sin^2\phi},
\\
\label{NSaxisym4}
&\qquad\qquad\qquad\partial_{\rho}(\rho^2\sin\phi\, u_{\rho})+\partial_{\phi}(\rho \sin\phi\,
u_{\phi}) =0,
\end{align}
where ${b} = u_\rho e_\rho + u_\phi e_\phi$. We consider the system on the domain 
\[
(\rho,\phi) \in (0,\infty)\times (0,\pi). 
\]

The natural boundary conditions for a smooth axisymmetric vector field $u$ are%
\EQN{
\label{u.bc}
\pd_\phi u_\rho = u_\th = u_\phi = \pd_\phi p = 0, \quad \text{at } \rho>0, \text{ and  at } \phi=0,\pi.
}

As we look for DSS solutions, we also impose the DSS
boundary conditions
\EQN{
\label{u.bc2}
 u_\rho(\rho,\phi) = \la u_\rho(\la \rho,\phi),\quad 
u_\phi(\rho,\phi) = \la u_\phi(\la \rho,\phi),\quad 
u_\th(\rho,\phi) = \la u_\th(\la \rho,\phi),
}
for some $\la>1$ to be chosen. We will only consider DSS solution $u \in H^1_\loc(\Om)$ for $\Om=\R^3 \setminus \{ 0\}$. By regularity theory, $u\in L^\infty_\loc$ and hence satisfies the bounds \eqref{SNS2} and \eqref{SNS3}.

\subsection{The stream function}
In this subsection we address the existence of a stream function $\psi$ such that $b=\curl(\psi e_\th)$. 

Since $b$ is a divergence-free vector field, $b$ can be written as a curl of some vector potential $F$. Recall in spherical coordinates, (see \cite[Appendix 2]{Bat})
\EQN{
\label{eq2.9}
  \nabla \times {\bf F} 
&=\frac{1}{\rho \sin \ph}\bke{\pd_\phi (F_\th \sin \phi) - \pd_\th F_\phi} e_\rho
+ \frac{1}{\rho}\bke{\frac 1{\sin \phi} \pd_\th F_\rho - \pd_\rho (\rho F_\th)} e_\phi \\
&\quad + \frac 1 \rho \bke{\pd_\rho (\rho F_\phi) - \pd_\phi F_\rho} e_\th.
}
Since $b$ is axisymmetric, we can choose axisymmetric $F$. Indeed, we can take $F= \psi e_\th$, $F_\rho=F_\psi=0$ and $F_\th=\psi=\psi(\rho,\phi)$ with
\EQN{
	\label{u.psi.relation}
	{b} = u_\rho e_\rho + u_\phi e_\phi 
	&=\curl (\psi e_\th) ,\\
	u_{\rho} =\frac 1{\rho\sin\phi} \pd_\phi  (\psi \sin \phi) ,&\qquad
	u_{\phi} = -\frac {1}{\rho } \pd_\rho  ( \rho \psi  ).
}

We now show the global existence of $\psi$.
For this purpose, we introduce the \emph{Stokes stream function} $\td\psi(\rho,\phi)$ (\cite[p.78]{Bat}) defined on $(\rho,\phi) \in (0,\infty)\times (0,\pi)$ by (\cite[(2.2.14)]{Bat})
\EQN{
\label{eq2.8}
\rho^2\sin\phi\, u_{\rho} = \pd_\phi \td\psi, \quad
-\rho \sin\phi\,
u_{\phi}= \pd_\rho\td \psi,
}
which exists by the divergence-free condition \eqref{NSaxisym4}. It relates to $\psi$ by
\EQN{\label{rel.psi}
\psi(\rho,\phi) = \frac 1{\rho\sin \phi}\, \td \psi(\rho,\phi) .
}
Since $u$ satisfies the bound \eqref{SNS2}, $u_\rho,u_\phi \in L^\infty(\Pi)$, $\Pi = (1,\la) \times (0,\pi)$, although $u_\rho$ may be discontinuous at $\phi=0,\pi$. Thus, $\td \psi\in W^{1,\infty}(\Pi)$ and is hence continuous in $\bar \Pi$. By \eqref{eq2.8}, 
$\pd_\phi \td\psi = \pd_\rho \td\psi =0$ when $\phi=0,\pi$. Now, as a consequence of the divergence free condition, we obtain
\[
0 = \int_{ |x|<\rho} \div u \, dx =  \int_{ |x|=\rho} u \cdot e_\rho dS_x = \int_0^{2\pi}\!\! \int_0^{\pi}  u_\rho \rho^2 \sin (\phi )d\phi d \th 
=2 \pi  \int_0^{\pi}  \pa_\phi \td\psi\, d\phi,
\]
and hence $\td \psi (\rho,0) = \td\psi(\rho,\pi)$ for any $\rho>0$. Since solutions to \eqref{eq2.8} are invariant under adding a constant, we can set $\td \psi (\rho,0) = \td\psi(\rho,\pi)=0$.
Therefore, we obtain 
\[
\psi(\rho,\phi) = \frac 1{\rho\sin\phi} \int_0^\phi \rho^2\sin \phi' u_\rho(\rho,\phi') d\phi' \in L^\infty(\Pi).
\]
The above argument uses the axisymmetry but not the DSS condition. 

\subsection{The equations for $(\psi,u_\th)$}
In this subsection, we reduce the equations \eqref{NSaxisym1}-\eqref{NSaxisym4} for velocity $u$ to a system for $\psi$ and $u_\th$. 
The boundary conditions for $u_{\rho}$ and
$u_{\phi}$ in \eqref{u.bc} are equivalent to the following boundary conditions for $\psi$
\EQN{
\label{psi.bc}
\pd_\rho (\rho\psi)|_{\phi=0,\pi}=\pd_\phi \bke{ \frac 1{\sin \phi} \pd_\phi (\psi \sin \phi )}\bigg |_{\phi=0,\pi}=
0.
}
Note that 
\[
\div (\psi e_\th) = 0, \quad \curl (\psi e_\th) = b,\quad (0<\phi<\pi).
\]
Since $b \in H^1_\loc(\Om)$, we have $\psi e_\th \in W^{1,6}_{\loc}(\Om)$ and hence $\psi e_\th $ is locally H\"older continuous in $\Om$. In particular,
\EQN{
\label{psi.bc1a}
\psi(\rho,0) = \psi(\rho,\pi) = 0.
}

On the other hand, the DSS boundary condition \eqref{u.bc2} implies
\EQN{
\label{psi.bc2a}
 \psi(\rho,\phi) = \psi(\la \rho,\phi).
}

In order to derive the equation for $\psi$, we first consider the equation for  
the vorticity. For axisymmetric flow $u$, by \eqref{eq2.9} the vorticity in the spherical coordinates is 
\EQ{
\om = \curl u = \om_\rho e_\rho + \om_\th e_\th + \om_\phi e_\phi
}
with
\EQN{
\label{om.formula}
\om_\rho =\frac{\pd_\phi (u_\th \sin \phi)}{\rho \sin \phi}, \quad
\om_\phi = - \frac{1}{\rho}\,{  \pd_\rho (\rho u_\th)} ,
\quad
\om_\th = \frac 1 \rho \bke{\pd_\rho (\rho u_\phi) - \pd_\phi u_\rho} .
}
With $b= \curl F$, $F=\psi e_\th$ and $\div F=0$,
\EQ{
\om_\th e_\th = \curl b = \curl \curl F = - \De F + \nb \div F = -\De (\psi e_\th) .
}
We introduce the operator $\hat A$ for $f=f(\rho,\phi)$,
\[
\hat{A} f = - \Delta_{\text{as}} f +\frac{f}{\rho^2\sin^2\phi}, \quad -\De (f e_\th) = (\hat{A}f)e_\th,
\]
where $\Delta_{\text{as}}$ is usual $\Delta$ restricted to axisymmetric functions,
\[
\Delta_{\text{as}} f=\frac{1}{\rho^2}\partial_{\rho}(\rho^2\partial_{\rho}f)
+\frac{1}{\rho^2\sin\phi}\partial_{\phi}(\sin\phi\partial_{\phi} f).
\]
Thus
\[
\om_\th e_\th = -\De (\psi e_\th) =(\hat A \psi)e_\th, \quad \om_\th = \hat A \psi.
\]

Recall the vorticity equation  
\[
- \De \om+ (u \cdot \nb) \om  =  ( \om \cdot \nb) u.
\]
The $\om_\th$ component satisfies
\EQ{
	\hat A \om_{\tht}
	+(u\cdot\na)\om_{\tht} + \frac{u_{\tht}\om_{\rho}}{\rho}+ \frac{u_{\tht}\om_{\phi}}{\rho}\cot\phi 
	=(\om\cdot\na)u_{\tht} + \frac{u_{\rho}\om_{\tht}}{\rho}+\frac{u_{\phi}\om_{\tht}}{\rho}\cot\phi.
}
Replacing $\om_\th = \hat A \psi$ and replacing $\om_\rho$ and $\om_\phi$ by \eqref{om.formula}
with $ d(u_\th) =  \om_\rho e_\rho  + \om_\phi e_\phi= \curl (u_\th e_\th)$, we get the equation for $\psi$
\begin{multline}
\label{psi.eq.full}
\hat{A}^2\psi +(b\cdot\na)\hat{A}\psi + \frac{u_{\tht}\pa_{\phi}(u_{\tht}\sin\phi)}{\rho^2\sin\phi}- \frac{u_{\tht}\pa_{\rho}(\rho u_{\tht})}{\rho^2}\cot\phi 
\\
=(d(u_\th) \cdot\na)u_{\tht} + \frac{\hat{A}\psi} {\rho} \left(u_{\rho}+u_{\phi}\cot\phi\right).
\end{multline}

The $u_\th$ equation \eqref{NSaxisym3} becomes
\EQN{
	\label{uth.eq}
	\hat{A}u_{\tht} +
	(b\cdot\na)u_{\tht} 
	+\frac{u_{\theta}}{\rho}  \left(u_{\rho}+u_{\phi}\cot\phi\right) %
	=    0.
}

Now, we consider the boundary conditions for $\psi$. Note that \eqref{psi.bc} and \eqref{psi.bc1a} are equivalent to 
\EQN{
\label{psi.bc1}
\psi|_{\phi=0,\pi}=\hat A \psi |_{\phi=0,\pi}=
0,
}
while \eqref{psi.bc2a} implies
\EQN{
\label{psi.bc2}
 \psi(\rho,\phi) = \psi(\la \rho,\phi), \quad 
\hat A \psi(\rho,\phi)  =\la^2 \hat A \psi(\la\rho,\phi).
}

The system \eqref{psi.eq.full}-\eqref{uth.eq} for $(\psi,u_\th)$ with the relations \eqref{u.psi.relation} and \eqref{om.formula} is self-contained with the boundary conditions \eqref{u.bc}, \eqref{u.bc2} for $u_{\tht}$ and \eqref{psi.bc1}, \eqref{psi.bc2} for $\psi$. It is the system for DSS, axisymmetric, steady Navier-Stokes flows which will be studied in the remaining part of this paper for a possible bifurcation of the Landau solutions.   

In the special case of an axisymmetric flow $u$ with no swirl, i.e, $u_\th=0$ and $u=b$, we no longer need \eqref{uth.eq}. In this case, our system of equations reduces to
\EQN{
\label{psi.eq}
\hat{A}^2 \psi  +
b \cdot \nb \hat{A}\psi - \frac{\hat{A}\psi}\rho(u_\rho+u_\phi \cot \phi) = 0
}
with the boundary conditions  \eqref{psi.bc1} and \eqref{psi.bc2} for $\psi$. 

\section{The linearization}
\label{sec3}
In this section we deduce the nonlinear equations for perturbations of Landau solutions and consider the linearization around Landau solutions. We will in particular study its kernel in the axisymmetric DSS class.

\subsection{Perturbation of Landau solutions}
Recall \eqref{Landau-sol}-\eqref{Landau-Psi} that the Landau solution $U^a$ with parameter $a>1$ is given by
\EQN{
\label{Landau-sol2}
U_\rho^a 
&= \frac 2{\rho}\bke{\frac {a^2-1}{(a-\cos \phi)^2} -1}
=\frac 1{\rho^2\sin\phi} \pd_\phi  (\Psi^a  \rho   \sin \phi)
\\
U_\phi^a 
&= - \frac {2\sin \phi} {\rho(a-\cos \phi)}
=- \frac {1}{\rho \sin\phi} \pd_\rho  (\Psi^a  \rho \sin \phi)
\\
\Psi^a & =  \frac {2\sin \phi }{a-\cos \phi}.
}
For the convenience, we drop the index $a$ in $U^a$, $U_\rho^a$, $U_\phi^a$ and $\Psi^a$ below. Note that these solutions are self-similar, axisymmetric, steady flows with no swirl, so that they solve \eqref{psi.eq} with $\psi = \Psi$ and $b= U$:
\EQN{
\label{Psi.eq}
\hat{A}^2 \Psi  +
U \cdot \nb \hat{A}\Psi - \frac{\hat{A}\Psi}\rho(U_\rho+U_\phi \cot \phi) = 0.
}

Denote a perturbation from a Landau solution by%
\EQ{
 \psi=\Psi+\xi,  \quad u = U + u_\th e_\th  + v.
}
The component $v=v_\rho e_\rho+ v_\phi e_\phi$ has no swirl and is 
determined by $\xi$
\EQ{
v(\xi) 
= \curl (\xi e_\th)=
\frac 1{\rho^2\sin\phi}\pa_{\phi}(\xi\rho\sin\phi) e_{\rho} -\frac 1{\rho\sin\phi}\pa_{\rho}(\xi\rho\sin\phi) e_{\phi} .
}
Subtracting the equation \eqref{Psi.eq} from \eqref{psi.eq.full}, the system \eqref{psi.eq.full}-\eqref{uth.eq} for the perturbation pair $(\xi, u_\th)$ can be written in the following form
\begin{equation}
\label{operator.form}
\left\{
\begin{split}
\hat {\sL} \xi
&= N_1(\xi,u_{\tht})\\
\hat {\sM} u_{\tht}   
&= N_2(\xi,u_{\tht}).
\end{split}
\right .
\end{equation}
Here, the linear operators $\hat\sL$ and $\hat\sM$ are defined by
\EQ{
\hat{\sL} \xi
&=\hat{A}^2\xi +(U\cdot\na)\hat{A}\xi + (v(\xi)\cdot\na)\hat{A}\Psi 
 - \frac{\hat{A}\xi}{\rho}(U_{\rho}+U_{\phi}\cot\phi)
-\frac {\hat{A}\Psi}{\rho}(v_{\rho}+v_{\phi}\cot\phi),
}
and
\EQ{
\hat{\sM} u_{\tht}
&=\hat{A} u_{\tht} + (U\cdot\na)u_{\tht}  +\frac{u_{\theta}}{\rho}(U_\rho+ U_{\phi}\cot\phi).
}
The non-linear mappings $N_1$ and $N_2$ are given by
\EQ{
N_1(\xi,u_{\tht})
=&  -(v(\xi)\cdot\na)\hat{A}\xi+\frac{\hat{A}\xi}\rho (v_\rho+v_\phi \cot \phi)\\
&-\frac{u_{\tht}\pa_{\phi}(u_{\tht}\sin\phi)}{\rho^2\sin\phi}
+ \frac{u_{\tht}\pa_{\rho}(\rho u_{\tht})}{\rho^2}\cot\phi 
+ (d(u_\theta)\cdot\na)u_{\tht},
}
and
\EQ{
N_2(\xi,u_{\tht})=-(v(\xi)\cdot\na)u_{\tht}
-\frac{u_\th}{\rho}(v_\rho+ v_\phi\cot\phi).
}

It is important to observe that the linear part of \eqref{operator.form}  is decoupled: 
$\hat{\sL}$ acts only on $\xi$ while $\hat{\sM}$ acts only on $u_\theta$. It will be convenient to call $\hat{\sL}$ the \emph{stream operator} and $\hat{\sM}$  the \emph{swirl operator}.
Note that, with $\curl U = \Om_\th e_\th$,
\EQN{\label{0904a}
(U_\rho+U_\phi \cot \ph) 
 = \frac 2{\rho} \frac {(a\cos \phi-1)}{(a-\cos \phi)^2},
\quad
\hat A  \Psi = \Om_\th = \frac{4(a^2-1)\sin \phi}{ \rho^2 (a-\cos \phi)^3}.
}

We will be looking for nonzero $(\xi,u_\th)$ satisfying \eqref{operator.form}
under the boundary conditions
\begin{align}
&\xi|_{\phi=0,\pi} = \hat{A}\xi|_{\phi=0,\pi} = 0 ,
\label{bc.xi.phi}\\
& \xi(\rho,\phi) = \xi(\la \rho,\phi), \quad 
\hat{A}\xi(\rho,\phi) = \la^2 \hat{A}\xi(\la \rho,\phi),
\label{bc.xi.rho}\\
&u_{\tht}|_{\phi = 0,\pi} = 0, 
\quad %
u_{\tht}(\rho,\phi) = \la u_{\tht}(\la \rho,\phi).
\label{bc.utht}
\end{align}
To this end, we look for nontrivial kernel of its linear part,
\EQN{ \label{eig.prob}
\begin{cases}
 \hat \sL \xi 
= 0 \\
 \hat \sM u_{\tht}   
=0
\end{cases}
}
under the same boundary conditions \eqref{bc.xi.phi}-\eqref{bc.utht}. This system includes two parameters $a$ and $\la$.

\subsection{Similarity variables} %

In this subsection, we introduce \emph{similarity variables} so that our system \eqref{operator.form} becomes periodic in the radial variable. It will enable us in next subsection to restrict \eqref{eig.prob} to Fourier subspaces of the radial variable and reduce our problem to one-variable problems.
 
Since $u$ is $\la$-DSS, both $\xi$ and $\rho u_{\tht}$ are $\la$-DSS of degree zero, in the sense that
\begin{equation*}
\left\{
\begin{split}
&\xi(\la \rho, \phi)=\xi( \rho, \phi), \\
&(\rho u_{\tht})(\la \rho, \phi)=\la\rho u_{\tht}(\la \rho, \phi) =(\rho u_{\tht})(\rho, \phi),
\end{split}
\right .
\quad \forall \rho>0, \ \forall \phi \in(0,\pi).
\end{equation*}

Introduce the new variable
\EQN{\label{ch.var}
\tau = \ln \rho, 
}
and define
the functions $\zeta$ and $\chi$ 
\EQ{
\zeta(\tau,\phi):=\xi(\rho,\phi),\quad 
\chi(\tau,\phi):=\rho u_{\tht}(\rho,\phi) .
}
They are both periodic in $\tau$ with period $\ln \la$. By the periodicity, we can restrict our domain to
\EQ{
(\tau,\phi) \in (0,\ln\la) \times (0,\pi).
}
Let
\[
\sL \zeta = \rho^4\hat\sL \xi, \quad
\sM \chi  = \rho^3\hat\sM u_{\tht}.
\]
In the new variables, the problem \eqref{eig.prob} on $(0,\ln\la) \times (0,\pi)$ becomes
\EQN{\label{eig.new.coord}
\begin{cases}
\sL \zeta = 0\\
\sM \chi =0 .
\end{cases}
}
The induced boundary conditions will be discussed in Section \ref{BC.tau}.

\subsection{Invariance of Fourier subspaces under the operators $\sL$ and $\sM$}

By the periodicity of $\zeta$ and $\chi$ in $\tau$, the linear system \eqref{eig.new.coord} can be considered in each Fourier mode, which leads to a family of 1D linear systems.
In other words, we consider \eqref{eig.new.coord} on each subfamily of functions %
\EQ{
\cF_n = \bket{ h(\phi) e^{in\si\tau} }, 
\quad \si = \frac{2\pi}{\ln \la}, 
\quad n\in \Z.
}
For this purpose, we first need the invariance of $\cF_n$ under the operations $\sL$ and $\sM$. In other words, if $\zeta$ and $\chi$ are in $\cF_n$, then 
\[
\sL \zeta \in \cF_n \quad\text{and}\quad
\sM \chi \in \cF_n.
\]
Note that functions of the form $f(\tau) e^{ik\phi}$ are not preserved under $\sL$ and $\sM$ whose coefficients depend on $\phi$.

\begin{proposition}\label{invariance}
The function spaces $\cF_n$, $n \in \Z$, are invariant under the linear operators $\sL$ and $\sM$.
\end{proposition}

Consider the linear operator $\sL$ first. We will decompose it in the form 
\EQ{
\sL = (A+B)A + C,
}
and then show the invariance of $\cF_n$ under the operators $A$, $B$ and $C$. 

\begin{lemma}\label{linear.op} The linear operator $\sL$ can be written as 
\bq
\sL = (A+B)A + C,
\eq
where 
\EQ{
A\zeta &= -(\pa_{\tau\tau}+\pa_{\tau}) \zeta -\pa_{\phi}\bke{\frac 1{\sin \phi}\pa_{\phi}(\zeta \sin \phi ) }
\\
B\zeta &=2\left(\frac{a^2-1}{(a-\cos\phi)^2}+1\right)\pa_{\tau}\zeta
-\frac {2\sin\phi}{(a-\cos\phi)}\pa_{\phi}\zeta
+\left(2-\frac{4a^2+2a\cos\phi-6}{(a-\cos\phi)^2}\right)\zeta
\\
C\zeta &= - \frac{12(a^2-1)\sin\phi}{(a-\cos\phi)^3} (\pa_{\phi}\zeta +\cot\phi \zeta)+\frac{12(a^2-1)\sin^2\phi}{(a-\cos\phi)^4}(\pa_{\tau}\zeta+\zeta).
}%
In particular, $A$ satisfies $A\zeta(\tau,\phi) = \rho^2\hat{A} \xi(\rho,\phi)$ where $\rho=e^{\tau}$ and $\zeta(\tau,\phi) =  \xi(\rho,\phi)$.
\end{lemma}

\begin{proof}
Recall that 
\bq
\hat\sL \xi = \hat{A}^2 \xi + (U \cdot\na )\hat{A} \xi + (v(\xi)\cdot\na) \hat{A}\Psi - \frac {\hat{A}\xi}{\rho}(U_{\rho} +U_{\phi}\cot \phi) - \frac {\hat{A}\Psi}{\rho}(v_{\rho}(\xi) + v_{\phi}(\xi)\cot \phi)
\eq
and
\EQ{
\hat{A}\xi = -\frac 1{\rho^2}\pa_{\rho}(\rho^2 \pa_{\rho} \xi) - \frac 1{\rho^2\sin\phi} \pa_{\phi}(\sin \phi \pa_{\phi} \xi) + \frac 1{\rho^2\sin^2\phi} \xi.
}
Under the change of variable $\rho= e^{\tau}$ with \eqref{ch.var},
$\pa_{\rho} = e^{-\tau}\pa_{\tau}$, and
\EQN{\label{operator.A}
 \rho^2 \hat{A} \xi
&= -\pa_{\rho}(\rho^2 \pa_{\rho} \xi) - \frac 1{\sin\phi} \pa_{\phi}(\sin \phi \pa_{\phi} \xi) + \frac 1{\sin^2\phi} \xi\\
&= -(\pa_{\tau\tau}+\pa_{\tau}) \zeta -(\cot \phi \pa_{\phi}+\pa_{\phi\phi})\zeta + \frac 1{\sin^2\phi} \zeta\\
&=-(\pa_{\tau\tau}+\pa_{\tau}) \zeta -\pa_{\phi}\bke{\frac 1{\sin \phi}\pa_{\phi}(\zeta \sin \phi ) } =: A \zeta .
}
Using for $F=F(\tau,\phi)$,
\[
-e^{k\tau}(\pa_{\tau\tau}+\pa_{\tau})(e^{-k\tau}F)
=-(\pa_{\tau\tau}+\pa_{\tau})F +2k\pa_{\tau}F - k(k-1)F,
\]
we have
\EQN{\label{Arho.commute}
e^{k\tau}A(e^{-k\tau}F)  = [A +2k\pa_{\tau} - k(k-1)]F.
}
Thus, with $k=2$ and $F=A\zeta$,
\EQ{
\rho^4 \hat{A}^2 \xi 
= e^{2\tau} A(e^{-2\tau}A\zeta)
= (A +  4\pa_{\tau}-2) A\zeta.
}

Now, we consider the terms $\rho^4 (U\cdot \na) \hat{A}\xi$ and $\rho^4 (v(\xi)\cdot\na) \hat{A}\Psi$.  
Using the gradient in the spherical coordinates 
\EQ{
\na = e_{\rho}\pa_{\rho}  +e_{\tht}\frac 1{\rho\sin\phi}\pa_{\tht}+ e_{\phi}\frac 1{\rho} \pa_{\phi},
}
we have for any $V(\phi)=V_\tau(\phi)e_\tau+V_\phi(\phi)e_\phi$ with $e_\tau = e_\rho$, 
\EQ{\rho^4 \left(\frac {V(\phi)}{\rho} \cdot \na\right) g(\rho,\phi)
&= \rho^3 (V(\phi)\cdot \na) g
= \rho^3 \bke{V_\tau(\phi) \pa_{\rho} + V_\phi(\phi)\frac 1{\rho} \pa_{\phi} } g\\
&= e^{2\tau} (V_\tau(\phi) \pa_{\tau} + V_\phi(\phi) \pa_{\phi} ) G\\
&= (V_\tau(\phi) \pa_{\tau} + V_\phi(\phi) \pa_{\phi} - 2V_\tau(\phi)) (e^{2\tau}G),
}
where $G(\tau,\phi)=g(\rho,\phi)$. As a consequence, plugging $V=\tilde{U}$, $g=\hat A \xi$ and $e^{2\tau}G=A\zeta$ where
\EQN{\label{td.U}
\tilde{U}(\phi)
:= \rho U
= 2\left(\frac{a^2-1}{(a-\cos\phi)^2}-1\right)e_{\rho}
-\frac {2\sin\phi}{(a-\cos\phi)}e_{\phi}
= \td U_\tau (\phi) e_\tau + \td U_\phi (\phi) e_\phi,
}
we get
\EQ{ \rho^4 (U \cdot\na) \hat{A} \xi 
&= (\tilde{U}_\tau \pa_{\tau}+\tilde{U}_\phi \pa_{\phi}) A \zeta 
-2\td U_\tau A\zeta\\ 
&=
\left(
2\left(\frac{a^2-1}{(a-\cos\phi)^2}-1\right)\pa_{\tau}
-\frac {2\sin\phi}{(a-\cos\phi)}\pa_{\phi}
\right) A \zeta -4\left(\frac{a^2-1}{(a-\cos\phi)^2}-1\right)A\zeta.\\
}
Similarly, we obtain
\EQ{
\rho^4 (v(\xi)\cdot\na) \hat{A}\Psi 
&=(\td v_{\tau} \pa_{\tau} + \td v_{\phi} \pa_{\phi} 
- 2\td v_{\tau}) (A\Psi )\\
&=[( \pa_{\phi}\zeta +\zeta\cot\phi ) \pa_{\tau} 
 -(\pa_{\tau}\zeta+\zeta) \pa_{\phi} 
- 2( \pa_{\phi}\zeta +\zeta\cot\phi )] A\Psi. 
}
Here, $\td v=\td v_{\tau}e_{\tau}+\td v_{\phi}e_{\phi}$ is defined by
\EQ{
\td v(\zeta)(\tau,\phi)
&= \rho v(\xi)(\rho,\phi) \\
&=\rho\left( 
\frac 1{\rho^2\sin\phi} \pa_{\phi}(\xi\rho\sin\phi)e_{\rho}
-\frac 1{\rho\sin\phi} \pa_{\rho}(\xi\rho\sin\phi)e_{\phi}\right)\\
&=
( \pa_{\phi}\zeta +\zeta\cot\phi ) e_{\tau}
-(\pa_{\tau}\zeta+\zeta) e_{\phi}.
}

Finally, the remaining terms are rewritten as 
\EQ{
&\rho^4 \left(-\frac{\hat{A}\xi}{\rho}(U_{\rho} + U_{\phi}\cot\phi)-\frac{\hat{A}\Psi}{\rho}(v_{\rho}(\xi)+v_{\phi}(\xi)\cot\phi) \right)\\
&= - A\zeta(\td U_{\tau} +\td U_{\phi}\cot \phi)
- A\Psi (\td v_\tau + \td v_\phi \cot\phi) \\
&=  -\frac{2(a\cos\phi -1)}{(a-\cos\phi)^2}A\zeta 
- (\pa_{\phi}\zeta 
 -\cot \phi\pa_{\tau}\zeta)A\Psi.
}
Note that (see also \eqref{0904a})
\EQ{
A\Psi = \rho^2 \hat{A}\Psi  =  \rho \bke{\pd_\rho (\rho U_\phi) - \pd_\phi U_\rho} = 2\pd_\phi \bke{\frac {a^2-1}{(a-\cos \phi)^2} -1}
=  \frac {4(a^2-1 )\sin \phi}{(a-\cos \phi)^3}.
}

Collect all terms in $\rho^4\hat{\sL}$. In order to write $\rho^4\hat\sL$ in the form $(A+B)A+C$, we need
\EQ{
B
&=-2 + 4\pa_{\tau} 
+\tilde{U}_\tau \pa_{\tau}+\tilde{U}_\phi \pa_{\phi}  
-2\td U_\tau 
-\frac {2(a\cos\phi-1)}{(a-\cos\phi)^2}\\
&=2\left(\frac{a^2-1}{(a-\cos\phi)^2}+1\right)\pa_{\tau}
-\frac {2\sin\phi}{(a-\cos\phi)}\pa_{\phi}
 -4\left(\frac{a^2-1}{(a-\cos\phi)^2}-\frac12\right)
-\frac {2(a\cos\phi-1)}{(a-\cos\phi)^2}\\
&=2\left(\frac{a^2-1}{(a-\cos\phi)^2}+1\right)\pa_{\tau}
-\frac {2\sin\phi}{(a-\cos\phi)}\pa_{\phi}
+2 -\frac{4a^2+2a\cos\phi-6}{(a-\cos\phi)^2}
}
and
\EQ{
C
&=(\pa_{\tau}A\Psi- 3A\Psi) (\pa_{\phi} +\cot\phi )  
 -(\pa_{\phi}A\Psi -A\Psi\cot \phi)(\pa_{\tau}+1)  \\
&=- \frac{12(a^2-1)\sin\phi}{(a-\cos\phi)^3} (\pa_{\phi} +\cot\phi )-\bke{\pa_{\phi}\frac{4(a^2-1)\sin\phi}{(a-\cos\phi)^3} -\frac{4(a^2-1)\cos\phi}{(a-\cos\phi)^3}}(\pa_{\tau}+1)  \\
&=- \frac{12(a^2-1)\sin\phi}{(a-\cos\phi)^3} (\pa_{\phi} +\cot\phi )+\frac{12(a^2-1)\sin^2\phi}{(a-\cos\phi)^4}(\pa_{\tau}+1). 
}

This completes the proof of the lemma.
\end{proof}

\begin{lemma}\label{invariance.sL}
The function spaces $\cF_n$, $n \in \Z$, are invariant under the operators $A$, $B$, $C$ and hence $\sL$. The restriction $\sL_n$ of $\sL$ on $\cF_n$ in the sense of 
\EQ{
	\sL(h(\phi)e^{in\si\tau})=(\sL_nh)(\phi)e^{in\si\tau}
}
is given by
\[
\sL_n = (A_n+B_n)A_n + C_n,
\]
where 
\EQN{\label{op.ABCn}
	A_n h
	&=((n\sigma )^2- in\sigma )h
	-\pa_{\phi}\bke{\frac 1{\sin \phi}\pa_{\phi}(h \sin \phi ) }
	\\
	B_nh 
	&=\left(
	2 i n\sigma\left(\frac{a^2-1}{(a-\cos\phi)^2}+1\right)
	-\frac {2\sin\phi}{(a-\cos\phi)}\pa_{\phi}
	+2 -\frac{4a^2+2a\cos\phi-6}{(a-\cos\phi)^2}
	\right)h\\
	C_nh
	&=\left(- \frac{12(a^2-1)\sin\phi}{(a-\cos\phi)^3} \pa_{\phi} +\frac{12i n\sigma(a^2-1)\sin^2\phi}{(a-\cos\phi)^4}
	+\frac{12(a^2-1)(1-a\cos\phi)}{(a-\cos\phi)^4}
	\right)h
}
are the restrictions of the operators $A$, $B$, and $C$ on $\cF_n$, respectively. %

\end{lemma}

\begin{proof}
For a function $\zeta\in \cF_n$, $\zeta(\tau,\phi)=h(\phi)e^{in\sigma\tau}$, we have
\EQ{
A [he^{in\sigma\tau}] 
&= -(\pa_{\tau\tau}+\pa_{\tau}) (he^{in\sigma\tau})
-\pa_{\phi}\bke{\frac 1{\sin \phi}\pa_{\phi}(he^{in\sigma\tau} \sin \phi ) }\\
&=\bke{((n\sigma )^2- in\sigma )h
-\pa_{\phi}\bke{\frac 1{\sin \phi}\pa_{\phi}(h \sin \phi ) }}e^{in\sigma\tau}\\
&= (A_nh)e^{in\tau}.
}

Similarly, we find the restrictions $B_n$ and $C_n$,
\EQ{
B [he^{in\sigma\tau}]
&=\left(
2 i n\sigma\left(\frac{a^2-1}{(a-\cos\phi)^2}+1\right)
-\frac {2\sin\phi}{(a-\cos\phi)}\pa_{\phi}
+2 -\frac{4a^2+2a\cos\phi-6}{(a-\cos\phi)^2}
\right)he^{in\sigma\tau}\\
&=(B_nh)e^{in\sigma\tau}
}
and
\EQ{
&C [he^{in\sigma\tau}]
\\
&=\left(- \frac{12(a^2-1)\sin\phi}{(a-\cos\phi)^3} (\pa_{\phi} +\cot\phi )+\frac{12(a^2-1)\sin^2\phi}{(a-\cos\phi)^4}(\pa_{\tau}+1)\right) he^{in\sigma \tau}\\
&=\left(- \frac{12(a^2-1)\sin\phi}{(a-\cos\phi)^3} \pa_{\phi} +\frac{12i n\sigma(a^2-1)\sin^2\phi}{(a-\cos\phi)^4}
+\frac{12(a^2-1)(1-a\cos\phi)}{(a-\cos\phi)^4}
\right)h e^{in\sigma \tau}\\
&=(C_nh) e^{in\sigma \tau}.
}

Obviously, each subspace $\cF_n$ is invariant under $A$, $B$ and $C$.
\end{proof}

Similar to getting the invariance of $\cF_n$ under the operator $\sL$, we prove its invariance under $\sM$.

\begin{lemma}\label{invariance.sL'} The spaces $\cF_n$, $n\in\Z$, are invariant under the operator $\sM$, and the restriction $\sM_n$ of $\sM$ on $\cF_n$ can be written as 
\EQ{
\sM_n= A_n+ E_n,
}
where $A_n$ is defined as in Lemma \ref{invariance.sL} and, with $\td U$ given by \eqref{td.U},
\EQN{\label{En}
E_n g =
\left(in\sigma(\tilde{U}_{\tau}+2)
+  \tilde{U}_{\phi}\pa_{\phi} 
+\td U_\phi \cot\phi
 \right) g.
}
\end{lemma}

\begin{proof}
Recall for $\chi = \rho u_\theta$,
\EQ{
\sM\chi =\rho^3\hat{\sM}u_{\tht}.
}
The right hand side can be written in the new variables $(\tau,\phi)$ as, using \eqref{Arho.commute},
\EQ{
\rho^3\hat{\sM}u_{\tht}
&=e^{\tau} \left(A+\tilde{U}_{\tau} \pa_{\tau}+  \tilde{U}_{\phi}\pa_{\phi} 
+\td U_{\tau} + \td U_\phi \cot\phi\right)(e^{-\tau}\chi) \\
&=\left(A +2\pa_{\tau} +\tilde{U}_{\tau} \pa_{\tau}+  \tilde{U}_{\phi}\pa_{\phi}
 + \td U_\phi \cot\phi\right)\chi 
\\
&= (A+E)\chi,
}
where $E=(\tilde{U}_{\tau}+2) \pa_{\tau}+  \tilde{U}_{\phi}\pa_{\phi} + \td U_\phi \cot\phi$.

By Lemma \ref{invariance.sL}, $\cF_n$ is invariant under the operator $A$ with its restriction given by $A_n$. Therefore, it is enough to show its invariance under $E$. 
Since
\EQ{
E(g(\phi)e^{in\sigma\tau}) 
&=\left(in\sigma(\tilde{U}_{\tau}+2)
+  \tilde{U}_{\phi}\pa_{\phi} 
+\td U_\phi \cot\phi
 \right) g(\phi) e^{in\sigma\tau},
}
$\cF_n$ is invariant under 
$E$ and the restricted operator $E_n$ is given as in \eqref{En}. 
\end{proof}

By Lemma \ref{invariance.sL} and Lemma \ref{invariance.sL'},  Proposition \ref{invariance} holds true. It reduces the system of equations \eqref{eig.new.coord} with two variables to a family of systems with one variable $\phi\in (0,\pi)$
\EQN{\label{eig.cF_n}\begin{cases}
\sL_n h = (A_n+B_n) A_n h + C_n h = 0  \\
\sM_n g = (A_n +E_n) g = 0. 					
\end{cases}}
When we study them, we keep the parameter $a$ but replace $\la$ by $\si$.

\subsection{Induced boundary conditions}
\label{BC.tau}
In the previous subsection, we get a family of linear systems of ordinary differential equations \eqref{eig.cF_n}. Now, we find the corresponding boundary conditions. 

For $\zeta(\tau,\phi) =\xi(\rho,\phi)$,  the boundary condition \eqref{bc.xi.phi}-\eqref{bc.xi.rho} for $\xi$ becomes
\EQN{\label{zeta.bc3}
	\zeta|_{\phi=0,\pi}=A \zeta |_{\phi=0,\pi}=0, 
}
\EQN{
	\label{zeta.bc2}
\zeta(\tau,\phi) = \zeta(\tau+\ln\la,\phi).
}
If $\zeta\in \cF_n$, i.e., $\zeta=h(\phi)e^{in\si\tau}$, the boundary condition \eqref{zeta.bc3}-\eqref{zeta.bc2} reduces to %
\EQN{\label{zeta_n.bc}
	h|_{\phi=0,\pi}=A _0 h |_{\phi=0,\pi}=0,
}
which implies $A_n h|_{\phi=0,\pi}=0$, where $A_n$, $n\in\Z$, is defined in Lemma \ref{invariance.sL}.

\medskip
On the other hand, for $\chi(\tau,\phi)=\rho u_{\tht}(\rho,\phi)$, the boundary condition \eqref{bc.utht} implies 
\EQN{\label{bc.chi}
	\chi|_{\phi=0,\pi} = 0,\quad 
	\chi(\tau,\phi) = \chi(\tau+\ln\la,\phi).
}

If $\chi\in \cF_n$, $\chi(\tau,\phi) = g(\phi) e^{in\si\tau}$, then \eqref{bc.chi} reduces to
\EQN{\label{bc.g}
	g|_{\phi=0,\pi} = 0. 
}

\subsection{Function spaces for the operators $\sL_n$ and $\sM_n$} \label{ftnsp}

We have considered $\sL_n$ and $\sM_n$ as differential operators. We now consider their domains and ranges. The base space is $X_0=L^2((0,\pi),\sin\phi \,d\phi)$, which is
 the Hilbert space equipped with the natural inner product:
\EQN{
(g,f)_{X_0} = \int_0^\pi g(\phi)\overline{f(\phi)} \sin\phi \,d\phi.
}
We will also use an $a$-dependent inner product
\EQN{\label{X0a}
(g,f)_{X_0^a} = \int_0^\pi g(\phi)\overline{f(\phi)} (a-\cos\phi)^2 \sin\phi \,d\phi.
}
They are equivalent but the constant depends on $a$.

Define the space $X_1$ by
\EQN{\label{spX1}
X_1 = \left\{g\in L^1_\loc(0,\pi) : \norm{g}_{X_1}^2 = \int_0^\pi \bke{|g'(\phi)|^2 + \frac{|g(\phi)|^2}{\sin^2\phi}} \sin\phi d\phi \,<\infty \right\},
}
which will work as the domain of the operator $\sM_n$, $n\in \Z$. 
Obviously, $X_1\subset X_0$. Furthermore, any function in $X_1$ is continuous on $(0,\pi)$ and vanishes at the boundary.  

\begin{lemma}\label{prop.sp.X1}
	If $g\in X_1$, then $g\in C_0([0,\pi])$. More precisely, $g$ satisfies
	\[
	\norm{g}_{C(0,\pi)} \lec \norm{g}_{X_1}, \quad 
	\lim_{\phi \to 0_+} g(\phi) = 0\quad\text{and}\quad
	\lim_{\phi \to \pi_-} g(\phi) = 0. 
	\]
\end{lemma}

\begin{proof}
By the change of variable $t = \ln|\csc \phi- \cot \phi|$, $dt = \frac 1{\sin\phi} d\phi$, and $G(t)= g(\phi)$, we have 
\[
\norm{g}_{X_1}^2=\int_0^\pi \bke{|g'(\phi)|^2 + \frac{|g(\phi)|^2}{\sin^2\phi}} \sin\phi d\phi
= \int_{ \R} |G'(t)|^2 + |G(t)|^2 dt = \norm{G}_{H^1(\R)}^2. 
\]	
By the Sobolev embedding, $G$ and hence $g$ are bounded and continuous, with
\[
\norm{g}_{C(0,\pi)} =\norm{G}_{C(\R)} \lec \norm{G}_{H^1(\R)} = \norm{g}_{X_1}.
\]
Furthermore, $\lim_{\phi \to 0_+,\pi_-}g(\phi) = \lim_{t \to -\infty,\infty}G(t)=0$.
\end{proof}

We now consider the operators $\sM_n$ and start with $\sM_0$.

\begin{lemma}\label{th3.6}
The weight $\om(\phi)=(a-\cos\phi)^2$ satisfies
\[
b[g,f]=
\int_0^\pi (\sM_0 g)\bar f(\phi) \om(\phi) \sin \phi \,d\phi =\overline{ b[f,g]},
\]
for any $f,g \in C^2([0,\pi])\cap X_1$.
It is the unique choice up to a constant factor.
\end{lemma}
\begin{proof}
Recall $\sM_0=A_0+E_0$,
\[
\sM_0 g = - \pd_\phi \bke{ \frac 1{\sin \phi} \pd_\phi(g\sin \phi )}
+ \frac {\td U_\phi}{\sin \phi} \pd_\phi(g\sin \phi ).
\]
Thus, using Lemma \ref{prop.sp.X1},
\[
b[g,f] = \int_0^\pi  \frac 1{\sin \phi} \pd_\phi(g\sin \phi )
(\pd_\phi + \td U_\phi)\bke{ \bar f \sin \phi \,\om} d \phi.
\]
We would have
\[
b[g,f]
=\int_0^\pi  \frac 1{\sin \phi} \pd_\phi(g\sin \phi )
\pd_\phi \bke{ \bar f \sin \phi} \om\, d \phi,
\]
which is symmetric,
if $\td U_\phi \om + \om'=0$. Since $\td U_\phi =-\frac{2\sin \phi}{a -\cos \phi}$, $\om(\phi)=C(a-\cos\phi)^2$.
\end{proof}

This lemma motivates the definition of the space $X_0^a$
as $b[g,f]=(\sM_0g,f)_{X_0^a}$. It also allows us to consider $\sM_0$ as a linear operator from $X_1$ to its dual space $X_1'$ by
\[
(\sM_0 g)(f) = \sB[g,f],\quad \forall f,g\in X_1,
\]
where the bilinear form $\sB$ is  
\[
\sB[g,f] := \int_0^{\pi} 
\pa_{\phi}(g\sin\phi) \pa_{\phi}(\bar{f}\sin \phi) \frac{(a-\cos\phi)^2}{\sin\phi}\, d\phi.
\]
Note that $\sB[g,g] \lec a^2 \norm{g}_{X_1}^2$.

Since $\td U_\tau \in L^\infty(0,\infty)$, the difference
\[
\sM_n -\sM_0 =((n\si)^2 -in\si + in\si (\td U_\tau +2))I 
= ((n\si)^2  + in\si (\td U_\tau +1))I 
\]
is also well-defined from $X_1$ to $X_1'$: For each $g\in X_1$, $(\sM_n-\sM_0)g$ is in $X_1'$, mapping any $f\in X_1$ into
\[
((\sM_n -\sM_0)g)(f) 
:=((\sM_n -\sM_0)g,f)_{X_0^a}
= ((n\si)^2g -in\si g + in\si (\td U_\tau +2)g, f)_{X_0^a}\in \CC.
\]
Similar to $\sM_0$, the element $g$ in the kernel of $\sM_n:X_1 \to X_1'$, $n\in \Z$, satisfy $\sM_n g = 0$ in the sense of
\[
(\sM_n g)(f) = (\sM_0 g)(f) +((\sM_n -\sM_0)g)(f)  =0, \quad\forall f\in X_1. 
\]

\medskip

To find a suitable domain and range for the operator $\sL_0$, we first consider the operator $A_0$. It can be defined on $X_1$ with its value in $X_1'$,
\EQ{
(A_0h)(f) 
&=\sB_0[h,f]
:= \int_0^\pi \pa_{\phi}(h\sin\phi)\pa_{\phi}(\overline{f\sin\phi)} \,\frac{1}{\sin\phi}\, d\phi.
}
for any $h,f\in X_1$. Tracking the proof of Lemma \ref{th3.6} with $\td U_\phi$ replaced by zero and $\om(\phi)=1$, we get $(A_0h)(f) = (A_0h,f)_{X_0}$, (not $X_0^a$), for any $h,f \in C^2([0,\pi])\cap X_1$. %

On the other hand, the operators $B_0:X_1\to X_1'$ and $C_0:X_1\to X_1'$ are also well-defined by
\[
(B_0h)(f) = (B_0h, f)_{X_0} , \quad 
(C_0h)(f) = (C_0h, f)_{X_0}, \quad \forall f,h\in X_1. 
\]

Therefore, $\sL_0h
= (A_0+B_0)(A_0h) + C_0h$ can be defined on 
\[
X_3 = 
\{h \in X_1: \ A_0 h \in X_1\},
\]
and $\sL_0 : X_3 \to (X_1)'$ can be defined as
\EQN{\label{sL_0.weakdef}
(\sL_0 h)(f)
= \sB_0[A_0h,f] + (B_0(A_0h), f)_{X_0} + (C_0h, f)_{X_0}, 
}
for any $h$ in $X_3$ and $f$ in $X_1$. Then, 
\[
\sL_0 h = 0
\] holds in the sense that $h\in X_3$ satisfies 
\[
(\sL_0h)(f) = 0, \quad \forall f\in X_1.
\]
We note that $h\in X_3$ implies $A_0h\in X_1$ and $h\in X_1$, and therefore it achieves the boundary conditions \eqref{zeta_n.bc} by Lemma \ref{prop.sp.X1}.

Finally, defining $A_n-A_0$, $B_n-B_0$, and $C_n-C_0$ on $X_1$ by
\EQ{
((A_n-A_0)h)(f)
&=((A_n-A_0)h,f)_{X_0},\\
((B_n-B_0)h)(f)
&=((B_n-B_0)h,f)_{X_0},\\
((C_n-C_0)h)(f)
&=((C_n-C_0)h,f)_{X_0}
}
for any $f$ and $h$ in $X_1$, we can easily check that $A_n$, $B_n$, and $C_n$ map from $X_1$ to $X_1'$. Furthermore, since 
\[
\norm{(A_n-A_0)h}_{X_1} = \norm{((n\si)^2 - in\si)h}_{X_1} \leq (|n\si|^2 + |n\si|) \norm{h}_{X_1},
\] 
for $h\in X_1$, $A_0 h \in X_1$ is equivalent with $A_n h\in X_1$ for all $n\in \Z$. This implies that $\sL_n:X_3 \to X_1'$ is well-defined. %

\begin{remark}$\sM_0$ is self-adjoint  over $X_0^a$ while $A_0$ is self-adjoint over $X_0$ since $\sB$ and $\sB_0$ are symmetric. Hence their eigenvalues are all real.
\end{remark}

\begin{remark}The eigenvalue problems $\sL_n h = \mu h $
and $\sM_ng = \mu g$ correspond to $\hat{\sL}_n\xi =\mu \rho^{-4} \xi$ and $\hat \sM_n u_\th = \mu \rho^{-2} u_\th$. The weights $\rho^{-4}$ and $\rho^{-2}$ are needed to fit the scaling property of DSS perturbations.

\end{remark}

\section{Zero and purely imaginary eigenvalues of swirl operators}
\label{sec4}
In this section, we prove that the trivial solution $g=0$ is the only solution $g\in X_1$ to 
\EQ{\begin{cases}
\sM_ng = \mu g,\\
g|_{\phi=0,\pi} = 0,
\end{cases}}
for either $\mu=0$ or $i\mu \in \R$, and for all $n\in \Z$, $a>1$ and $\si>0$.
Note that $\sM_n$ has a dependence on $a>1$ and $\si>0$ for $n\not=0$, and $\sM_0$ depends on $a$ only.
Recall that $\sM_n: X_1 \to X_1'$ is defined for $g,f\in X_1$ by
\EQ{
(\sM_ng)(f) &= (\sM_0g)(f) + ((\sM_n-\sM_0)g)(f) \\
&= \int_0^{\pi} 
\pa_{\phi}(g\sin\phi) \pa_{\phi}(\bar{f\sin \phi}) \frac{(a-\cos\phi)^2}{\sin\phi} d\phi
+\big((n\si)^2g +  in\si(\tilde{U}_{\rho}+1)g,f \big)_{X_0^a},
}
where $(,)_{X_0^a}$ is defined in \eqref{X0a}.
If such solution $g$ exists, it satisfies the zero boundary condition by Lemma \ref{prop.sp.X1}. The following is the main theorem of this section.

\begin{theorem}\label{eig.sM.thm} For any $a>1$, $\si>0$, the operator $\sM_n = A_n+E_n: X_1 \to X_1'$ for any $n\in \Z$ does not have zero eigenvalue, nor any purely imaginary eigenvalue. 
\end{theorem}

\begin{proof}
Fix $a$, $\si$, and $n$.
Any eigenfunction $g\in X_1$ of $\sM_n$ with eigenvalue $\mu$ satisfies
 $(\sM_n g)(g) = \mu (g,g)_{X_0^a}$. Suppose that $\mu$ is either zero or purely imaginary. 
Taking the real parts, we get
\begin{align*}
\Re ((\sM_ng)(g)) = 0.
\end{align*}
However,
\EQ{%
0=\Re ((\sM_ng)(g))
=\int_0^{\pi} 
|\pa_{\phi}(g\sin\phi)|^2  \frac{(a-\cos\phi)^2}{\sin\phi} d\phi
+\big((n\si)^2g,g\big)_{X_0^a}.
}
Hence $g=0$ and is not an eigenfunction.
\end{proof}

\begin{remark} 
Theorem \ref{eig.sM.thm} implies that, if there is a bifurcation curve originating from a Landau solution in the class of DSS, axisymmetric steady flows along a zero eigenfunction of the linearized operator, then the swirl component of the eigenfunction is zero. In view of the nonlinear system \eqref{operator.form} for the perturbation $(\xi,u_\th)$, the solutions on the curve sufficiently close to the Landau solution must have zero swirl components.
\end{remark}

\section{Analysis of stream operators} \label{sec5}
In this section, we analyze the kernel of $\sL_0$ and the eigenvalues of $\sL_n$ both analytically and numerically, with the help of asymptotic analysis. These operators are defined in Lemma \ref{invariance.sL}. Since the Landau solutions \eqref{Landau-sol} is a continuous family with parameter $a$, one expects and can verify that $\sL_0$ has $\pa_a\Psi$ in its kernel. We will first prove that this is the only element in the kernel of $\sL_0$ up to a constant multiple. Then, we present numerical evidence for the non-existence of zero eigenvalue of $\sL_n$ when $n$ is a non-zero integer, and that there are no purely imaginary eigenvalues.%

\subsection{Linear operator $\sL_0$ and its kernel}
In this subsection, we consider the linear operator $\sL_0$
\EQ{
\sL_0 = (A_0+B_0)A_0 +C_0.
}
Note that it depends on the parameter $a\in (1,\infty)$ but not on $\si$. Since the Landau solutions \eqref{Landau-sol} is a continuous family with parameter $a$ of explicit solutions to (SNS) which are axisymmetric and self-similar, one expects and can verify that $\sL_0$ has $\pa_a\Psi$ in its kernel. 
In fact, $\pa_a\Psi$ is the unique eigenfunction up to a constant factor, which is reasonable in view of the rigidity result of \cite{Sverak2011} since the zero mode $n=0$ corresponds to minus one homogeneous functions in $\R^3$. Recall \eqref{Landau-sol2},
\[
\Psi(\psi)=\frac{2 \sin \phi}{a-\cos \phi}, \quad \pa_a\Psi=\frac{-2 \sin \phi}{(a-\cos \phi)^2}.
\]

\begin{theorem}\label{simple.eigenvalue} For any $1<a<\infty$, the kernel of $\sL_0: X_3 \to X_1'$ is spanned by $\pa_a\Psi$. There is no strictly generalized eigenfunction.
\end{theorem}

To prove this theorem, we do the following change of variables
\EQN{\label{variable.z}
\cos \phi = z, \quad
-\sin \phi \pa_z = \pa_{\phi}.
}
Then, we can write $\sL_0$ in a simpler form.

\begin{lemma}\label{sL_0.z} 
Let $z,\phi$ satisfy \eqref{variable.z}.
For $h\in C^4_{\loc}(0,\pi)$, we have
\EQN{\label{tdL0}
\frac 1{\sin\phi}\sL_0 h = \td L_0 H, \quad H(z) =\sin\phi \,h(\phi),
}
where the linear operator $\td L_0$ is defined by
\EQ{ 
\td L_0 H = ((1-z^2)H'+2zH-f_0H)''',\quad
f_0(z) = \frac{2(1-z^2)}{a-z}. 
}
\end{lemma}

\begin{remark}Here we understand \eqref{tdL0} in pointwise sense. 
Note that $\norm{h}_{X_1}^2=\int_{-1}^1 (dH/dz)^2dz$. Note $f_0(z) = \Psi(\phi)\sin \phi$, and $H_0(z) = \pd_a f_0(z) =\pd_a\Psi(\phi)\sin \phi$
will appear in \eqref{H0z.def}. 
\end{remark}

\begin{proof} %

Under the change of variable \eqref{variable.z}, $A_0 h$ (see \eqref{op.ABCn}) can be written as 
\EQN{\label{A_0.z}
A_0 h 
= -\frac{d}{d\phi}\bke{\frac1{\sin\phi}\frac{d}{d\phi}(h\sin\phi )}
= -\sqrt{1-z^2} H''(z).
}

This implies that
\EQN{\label{A_02.z}
A_0^2 h = A_0(A_0 h)
= -\sqrt{1-z^2} \pa_z^{(2)}(\sin\phi A_0h)
= \sqrt{1-z^2} ((1-z^2) H'')''.
}

In a similar way, $B_0A_0$ and $C_0$ can also  be written as 
\EQN{\label{B_0.z}
B_0(A_0 h)
&= -\frac{2\sin\phi}{a-\cos\phi}\pa_{\phi}(A_0 h) + V(A_0 h)\\
&= \frac{2\sin^2\phi}{a-\cos\phi}(-\sqrt{1-z^2} H'')' - V\sqrt{1-z^2} H''\\
&= \frac{2(1-z^2)}{a-z}\bke{-\sqrt{1-z^2}H'''+\frac{z}{\sqrt{1-z^2}}H''} -V\sqrt{1-z^2} H''
}
with $V(\phi)= 2 -\frac{4a^2+2a\cos\phi-6}{(a-\cos\phi)^2}=2 -\frac{4a^2+2az-6}{(a-z)^2}$, and
\EQN{\label{C_0.z}
C_0 h 
& = -\frac {12(a^2-1)}{(a-\cos\phi)^3}
\left(\sin\phi \, h' -\frac{1-a\cos\phi}{a-\cos\phi}h\right)\\
& = -\frac {12(a^2-1)}{(a-\cos\phi)^3}
\left(-\sin^2\phi \pa_z(\frac{H}{\sqrt{1-z^2}}) -\frac{1-az}{a-z}\frac{H}{\sqrt{1-z^2}}\right)\\
& = \frac {12(a^2-1)}{(a-z)^3}
\left(\sqrt{1-z^2} H' 
+\frac{H}{\sqrt{1-z^2}} \left(z
+\frac{1-az}{a-z}\right)\right)\\
&= \frac {12(a^2-1)\sqrt{1-z^2}}{(a-z)^3}
\left( H' 
+\frac{H}{a-z}\right).\\
}

Summing up \eqref{A_02.z}, \eqref{B_0.z}, and \eqref{C_0.z}, we can rewrite $\frac {1}{\sin\phi}\sL_0 h$ as
\EQ{
\frac {1}{\sin\phi}\sL_0 h
= (1-z^2)H''''-(4z+f_0)H''' -\frac{6(z^2-2az+1)}{(a-z)^2} H''
+ \frac{12(a^2-1)}{(a-z)^3}H' + \frac{12(a^2-1)}{(a-z)^4}H, 
}
which matches the right hand side of \eqref{tdL0}.
\end{proof}

Now, we prove that $\pa_a\Psi$ is the unique solution of $\sL_0 h = 0$ up to a constant factor.

\begin{proof}[Proof of Theorem \ref{simple.eigenvalue}]

If $h \in X_3$ is in the kernel of $\sL_0$, i.e., $\sL_0 h\equiv0$, we have $h \in C^\infty_\loc(0,\pi)$ by usual regularity theory. By Lemma \ref{sL_0.z}, $\sL_0 h(\phi) =0$ on $(0,\pi)$ is equivalent to $\td L_0 H(z)=0$ on $(-1,1)$ for $H(z) = \sin\phi \, h(\phi)\in C^\infty_\loc(-1,1)$ and $z= \cos\phi$. Then,
\begin{align}
\label{exp.tdL0}
&\td L_0 H(z)= ((1-z^2)H'+2zH-f_0H)'''=0 \nonumber\\
&\iff 
(1-z^2)H'+2zH-f_0H = a_0 + a_1z + a_2 z^2 \quad \text {on } (-1,1),
\end{align}
for some constants $a_0$, $a_1$, and $a_2$, where $f_0(z) = \frac{2(1-z^2)}{a-z}$. Furthermore, since $h\in X_3$,
the function $H(z)=\sin\phi \, h(\phi)$ satisfies $H\in C^1([-1,1])$, $H(-1)=H(1)=0$. Indeed, $H\in C_0([-1,1])$ follows from Lemma \ref{prop.sp.X1}. Moreover, applying Lemma \ref{prop.sp.X1} to $A_0h$, we have $A_0h\in C_0([0,\pi])$ and hence $\sqrt{1-z^2}H''(z) \in C_0([-1,1])$ by \eqref{A_0.z}. Then, writing $H'(z) = H'(0) + \int_0^z \frac{\sqrt{1-w^2} H''(w)}{\sqrt{1-w}\sqrt{1+w}} dw$, the integrability of $\frac{\sqrt{1-w^2} H''(w)}{\sqrt{1-w}\sqrt{1+w}}$ in $[-1,1]$ gives $H'\in C([-1,1])$. Now, taking limits on \eqref{exp.tdL0} as $z\to \pm 1$, we obtain
\EQ{
a_0 \pm a_1 + a_2 =0 \implies a_1=0, \ a_2=-a_0.
}
In other words, \eqref{exp.tdL0} becomes
\EQN{\label{eq1}
	(1-z^2)H' + 2zH -f_0H = a_0(1-z^2).
}
We can take derivative of \eqref{eq1} in $z$ to get
\[
(1-z^2)H'' +2H - f_0 H' -f_0' H = -2a_0 z.
\]
Taking again limits as $z\to \pm 1$ and using $f_0(\pm 1)=H(\pm 1)=0$ and $\sqrt{1-z^2}H''(z) \in C_0([-1,1])$, we get
\[
0 + 0 - 0 -0 - 0 = \mp 2a_0.
\]
Thus $a_0=0$. In other words, to solve $\td L_0 H=0$ under the given boundary conditions, it is enough to solve $\eqref{eq1}$ with $a_0=0$. 

We now look for an integration factor $k(z)$ so that
\EQN{\label{L0.factor1}
(1-z^2)H' + 2zH -f_0H =(1-z^2)k^{-1} \frac d{dz} (kH)
}
Since $f_0(z) = \frac{2(1-z^2)}{a-z}$,
\[
\frac{k'}k =  \frac{2z -f_0}{1-z^2} =  \frac{2z }{1-z^2} + \frac{2}{z-a}.
\]
Therefore, $k$ satisfies
\[
\ln k = \int \frac{2z }{1-z^2} + \frac{2}{z-a} dz =- \ln (1-z^2) + 2 \ln |z-a|+c,
\]
so that it can be chosen as
\[
k= \frac {(a-z)^2}{1-z^2}.
\]
This implies that the solutions of $\td L_0 H =0$ with $H(\pm1) =0$ is
\EQN{\label{H0z.def}
H = CH_0, \quad 
H_0(z) =-2k^{-1} = \frac {-2(1-z^2) }{(a-z)^2}
}
for some $C$. Since $H_0(\cos\phi)=\sin \phi \,\pa_a\Psi(\phi)$, we have $h=C \pa_a\Psi$, i.e., any solution $h\in X_3$ of $\sL_0 h=0$ is a multiple of $\pa_a \Psi$. 

Suppose now we have a generalized eigenfunction $h \in X_3$ satisfying 
\EQN{\label{sL0h=paPsi}
\sL_0 h = \frac12 \pa_a \Psi.
}
By Lemma \ref{sL_0.z}, $H(z) =  h(\phi)\sin \phi$ satisfies
\[
\td L_0 H =\frac 1{ \sin \phi} \sL_0 h =  \frac1{2 \sin \phi} \pa_a \Psi =    \frac1{2 \sin^2 \phi} H_0(z) = \frac {-1}{(a-z)^2}.
\]
Using the formula for $\td L_0$ in Lemma \ref{sL_0.z} and integrating three times, we get
\[
(1-z^2)H'+2zH-f_0H = G(z) := a_0 + a_1z + a_2 z^2 -(a-z)[ \ln (a-z)-1]
\]
for some constants $a_1,a_2,a_3$. By the argument in the first part of the proof, we have $G(\pm 1)=0$ and $G'(\pm 1)=0$. 
The conditions $G(\pm 1)=0$ give
\[
a_0+a_1+a_2 = (a-1)[\ln(a-1)-1], \quad
a_0-a_1+a_2 = (a+1)[\ln(a+1)-1],
\]
hence $2a_1=2+(a-1)\ln(a-1)-(a+1)\ln(a+1)$. The conditions $G'(\pm 1)=0$ give
\[
a_1+2a_2 =- \ln(a-1), \quad
a_1-2a_2 =- \ln(a+1),
\]
hence $2a_1=- \ln(a-1)- \ln(a+1)$. These two equations for $2a_1$ give
\[
f(a):=\frac2a+\ln(a-1)-\ln(a+1)=0.
\]
But $\lim_{a\to 1^+} f(a)=-\infty$, $\lim_{a\to \infty}f(a)=0$ and $f'(a)=\frac2{a^2(a^2-1)}>0$ for $a>1$. Hence $f(a)<0$ in $(1,\infty)$ and there is no solution  $h \in X_3$ of \eqref{sL0h=paPsi}.
This completes the proof of Theorem \ref{simple.eigenvalue}.
\end{proof}

\begin{remark} \label{sL_0.z2}
In view of \eqref{L0.factor1}, we can factorize $\td L_0$,
\EQN{
\td L_0 H(z) = \pd^3_z \bke{ (1-z^2)H_0 \pd_z \frac H{H_0}}
=  \pd^3_z \bket{ \frac{ (1-z^2)^2}{(a-z)^2} \pd_z  \bke{\frac {(a-z)^2}{1-z^2} H }}.
 }
\end{remark}

\subsection{Eigenvalues of $\sL_0$} 
In the following two subsections we study numerically the general eigenvalue problem for the linear operator $\sL_n$, $n \in \Z$. This is necessary even for the numerical
study of the zero eigenvalue of $\sL_n$ due to numerical errors.
The focus of our study is to examine whether the smallest real parts of eigenvalues are positive, except the zero eigenvalue of $\pd_a \Psi$. A positive result would imply that there is no nontrivial zero eigenfunction and no purely imaginary eigenvalue, and is an evidence of no bifurcation of Landau solutions.

Consider the general eigenvalue problem for the linear operator $\sL_n$:%
\EQ{\begin{cases}
\sL_n h =\mu h\\
h|_{\phi=0,\pi}=A_n h_{\phi=0,\pi} = 0.
\end{cases}}
In this problem, we look for an eigenvalue $\mu\in \mathbb{C}$ and an eigenfunction $h\in X_3$ such that for any $f\in X_1$, we have
\[
(\sL_n h)(f) = \mu(h,f)_{X_0}. 
\]
Recall the definitions of the space $X_0$, $X_1$, $X_3$ and $\sL_n:X_3\to X_1'$ in Section \ref{ftnsp}.

By the decomposition $\sL_n = (A_n+B_n)A_n + C_n$ in Lemma \ref{invariance.sL}, we can rewrite the eigenvalue problem as
\EQN{\label{eq5.10}
\begin{pmatrix}
I & -A_n\\
A_n +B_n & C_n
\end{pmatrix}
Y
=\mu
\begin{pmatrix}
0 & 0\\
0 & I
\end{pmatrix}
Y, \quad 
Y= \begin{pmatrix} A_n h \\  h\end{pmatrix}.
}
This formulation seems natural because the two components of $Y$ live in the same space $X_1$. It is convenient for the numerical study as it changes a fourth order system to second order. We will apply a finite difference scheme to a discretized version of \eqref{eq5.10}. As a stationary Navier-Stokes flow satisfying the bound \eqref{SNS2} has higher regularity \eqref{SNS3}, we can show sufficient regularity of the solution for the convergence of the finite difference scheme.

We consider two cases: one is for $n=0$ and the other is for $n\neq 0$. %
We study the case $n=0$ in this subsection. We will study the case $n \not = 0$ in next subsection.

We first describe our numerical observations for the case $n=0$:
\begin{enumerate}
\item The  smallest absolute value of eigenvalue is close to $0$,  and the second smallest is away from $0$. This agrees with Theorem \ref{simple.eigenvalue} that the eigenvalue $0$ of $\sL_0$ only has the eigenfunction $\pa_a\Psi$ up to a constant multiple. It also suggests that there is no strictly generalized eigenfunction $h(\phi)$ with $\sL_0 h = \pa_a\Psi$. 

\item All eigenvalues are \emph{real and (almost) nonnegative.}
\end{enumerate}

We cannot explain the second numerical observation above.
One might guess that $\sL_0$ is self-adjoint in $$L^2(0,\pi; w(\phi)\,d\phi)$$ for some weight function $w(\phi)$, but this is disproved by the following lemma.

\begin{lemma}
The bilinear form
\[
B(g,h) = \int_0^\pi g (\sL_0 h) w(\phi)\,d\phi,\quad g,h \in C^\infty_c(0,\pi),
\]
is not symmetric for any smooth weight $w(\phi) >0$ in $(0,\pi)$: For any such weight $w$, there are $g,h\in C^\infty_c(0,\pi)$ so that $B(g,h)\not= B(h,g)$.
\end{lemma}
\begin{proof}
Change variables
\[
z = \cos \phi, \quad G(z) = g(\phi)\sin \phi, \quad H(z)= h(\phi)\sin \phi, \quad S(z) = \sin \phi.
\]
By Lemma \ref{sL_0.z}  and Remark  \ref{sL_0.z2},
\EQ{
B(g,h) &= \int_{-1}^1 \frac {G}S \, \bke{S \td L_0 H}\frac w S\, dz
=  \int_{-1}^1 \frac {Gw}S  \, \td L_0 H\, dz
\\
&= \int_{-1}^1 \frac {Gw}S \,\pd_z^3  \bke{ Q \pd_z (k H)}\, dz,
}
where
\[
Q= \frac{ (1-z^2)^2}{(a-z)^2}, \quad k = \frac {(a-z)^2}{1-z^2} .
\]
Suppose $w=kS W$ for some $W>0$. Then
\[
B(g,h) =  \int_{-1}^1 kGW \,\pd_z^3  \bke{ Q \pd_z (k H)}\, dz
= \int_{-1}^1 kG  \, L (kH)\, dz = \int_{-1}^1 L^*(kG ) \, kH\, dz,
\]
where
\[
Lu = W\pd_z^3  \bke{ Q \pd_z u}
= W\bke{Q u^{(4)} + 3Q' u''' + 3 Q''u'' + Q''' u'},
\]
\EQ{
L^*u &=\pd_z \bke{Q\pd_z^3  \bke{ W  u}}
\\ &= QW  u^{(4)} +(Q'W+ 4QW') u''' + (3 Q'W'+6QW'')u''
\\ & \qquad  \qquad \quad + (3Q'W''+4QW''') u' + (QW''')' u.
}
For $B(g,h)$ to be symmetric, we need $L=L^*$, and hence their coefficients should match. Matching $u'''$ coefficients,
\[
3Q'W = Q'W+ 4QW', \quad \text{i.e.}\quad
WQ'= 2 QW'.
\]
We get $2W'/W = Q'/Q$, $W^2=cQ$. We may choose $c=1$ and hence 
\[
W=Q^{1/2} = \frac{ 1-z^2}{a-z}  =  z+a -  \frac{a^2- 1}{a-z}.
\]
Matching $u$ coefficients, $0= (QW''')' $, hence
\[
c = Q W''' = W^2 \frac{-6(a^2-1)}{(a-z)^4},
\]
which is a contradiction. The lemma is proved.
\end{proof}

We formulate a conjecture.

\begin{conjecture}\label{conj.sL0} For all $a>1$, 
all nonzero eigenvalues of the linear operator $\sL_0$ are real and positive.
\end{conjecture} 

In the finite difference scheme applied to the second order ODE system \eqref{eq5.10}, we first introduce a finite-dimensional approximate eigenvalue problem for the operator $\sL_0$  
\EQN{\label{ex.sL0}
\begin{pmatrix}
I & -A_0\\
A_0 +B_0 & C_0
\end{pmatrix}
Y
=\mu
\begin{pmatrix}
0 & 0\\
0 & I
\end{pmatrix}
Y.
}
This finite-dimensional problem is obtained by approximating the eigenvalue problem at a finite number of points $\{\phi_k\}_{k=0}^{N+1}$ on the interval $[0,\pi]$ defined by %
\EQ{
\phi_k = \frac{\pi}{N+1}k =: \del k, \quad k=0,\cdots,N+1.
}
The boundary conditions with $\phi_0=0$, and $\phi_{N+1}=\pi$ give
\EQN{\label{bc.matrix}
h(\phi_0)=h(\phi_{N+1}) =0,\quad
A_0h(\phi_0)=A_0h(\phi_{N+1})=0. 
}
The first and the second derivatives are approximated by
\EQN{\label{der.hd}
h'(\phi_k) \sim \frac{h_{k+1}-h_{k-1}}{2\del},\quad
h''(\phi_k) \sim \frac{h_{k+1}-2h_{k}+h_{k-1}}{\del^2}, \quad
\forall k=1,\cdots,N
}
where $h_k=h(\phi_k)$. Based on these, the operators $A_0$, $B_0$ and $C_0$ can be expressed as a $N\times N$ matrix by \eqref{der.hd} acting on $(h_1,\ldots,h_N)^T$, and \eqref{ex.sL0} becomes an eigenvalue problem for a $2N\times 2N$ matrix with the eigenvector $Y \in \CC^{2N}$.

Now, we find the eigenvalue $\mu$ in \eqref{ex.sL0} by the assistance of MATLAB using commands \texttt{eig} and \texttt{eigs}. Tables \ref{min0} and \ref{2min0} are the tables of the first and second minimum of real part of eigenvalues of $\sL_0$, respectively. Recall that $\sL_0$ depends on $a$ but not on $\si$.  The notation 4.3375e+06 means $4.3375\cdot 10^{+06}$.%

\begin{table}[H]
\caption{Minimum of real parts of eigenvalues of $\sL_0$}\label{min0}
\begin{center}
\begin{tabular}{|l||*{6}{c|}} 
\hline
\diagbox{$N$}{$a$} & 1.001 & 1.01 & 1.1 & 1.2 & 2\\
\hline
100 & -4.3375e+06 & -526.5826 & -0.4929 & -0.1113 & -0.0066\\
320 & -4.9314e+04 & -19.8387 & -0.0465 & -0.0108 & -6.5386e-04\\
640 & -5.9662e+03 & -4.0271 & -0.0116 & -0.0027 & -1.6395e-04\\
900 & -0.24404e+03 & -1.9419 & -0.0059 & -0.0014 & -8.2981e-05\\
\hline
\end{tabular}
\end{center}
\end{table}

\begin{table}[H]
\caption{Second minimum of real parts of eigenvalues of $\sL_0$}\label{2min0}
\begin{center}
\begin{tabular}{|l||*{6}{c|}} 
\hline
\diagbox{$N$}{$a$}& 1.001 & 1.01 & 1.1 & 1.2 & 2\\
\hline
100 & 11.9690& 11.7248& 18.7715& 20.3521& 23.0242\\
320 & 11.9535& 13.3592& 19.1610& 20.4829& 23.0448\\
640 & 11.9611& 14.9310& 19.1929& 20.4937& 23.0465\\
\hline
\end{tabular}
\end{center}
\end{table} 

As mentioned at the beginning of this subsection, we find that all eigenvalues of $\sL_0$
for the specific choices of $a$ and $N$ on the tables are real-valued. To further support the observation about real eigenvalues, we consider more candidates $N=640$ and additional $a=10, 10^2, 10^4, 10^6$ and still obtain all real eigenvalues. In this investigation, we used the command \texttt{eig} to obtain an array of all ($2N$) eigenvalues, \texttt{imag} to extract the imaginary parts, and \texttt{min} and \texttt{max} to see that, indeed, all imaginary parts are zero. (MATLAB returned exact 0, not something like $10^{-6}$.) We then obtain the first and second minimums of the real-parts of all eigenvalues of $\sL_0$, using the command \texttt{eig}. 

For comparison, we also applied the above procedure to $\sL_1$, and found that $\sL_1$ has eigenvalues with non-zero imaginary part. Thus the observation that $\sL_0$ has only real eigenvalues should not be due to code error.

In the tables, we do not have $0$ as an eigenvalue, which is the known eigenvalue with the eigenfunction $\pa_a\Psi$ in theory. Instead, the minimum of real part of eigenvalues of $\sL_0$ has a negative value, which approaches $0$ very quickly as $N$ increases (at least for $a\ge1.1$). Also, the second minimum is quite far from the first minimum. %
Thus, we guess that the true eigenvalue corresponding to the negative real part is $0$. To make sure of this, we compute the cosine of the angle between the approximated eigenfunctions and the true eigenfunction $\pa_a\Psi$. As we see in Table \ref{angle}, the cosine of the angle is almost $1$, which means the approximated eigenfunction is almost the same as the true eigenfunction.    

\begin{table}[H]
\caption{Cosine of angle between the approximated eigenfunction and $\pa_a\Psi$}\label{angle}
\begin{center}
\begin{tabular}{|l||*{6}{c|}}
\hline
\diagbox{$N$}{$a$}& 1.001 & 1.01 & 1.1 & 1.2 & 2\\
\hline
640 & 0.9988 & 1 & 1& 1 & 1\\
\hline
\end{tabular}
\end{center}
\end{table}

Moreover, we have only one eigenvalue which is close to zero. It is an evidence that $\pa_a\Psi$ is the unique eigenfunction of $\sL_0$. Also, we observe that all the other eigenvalues are positive, which provides the evidence of Conjecture \ref{conj.sL0}.

\subsection{Eigenvalues of $\sL_n$, $n\not =0$}
In this subsection we study the eigenvalues of $\sL_n$ in the second case  $n\not =0$.

First, we claim that it is enough to consider $n\in \N$, i.e., $n>0$, because $n$-mode and $-n$-mode have the same real-parts of the eigenvalues. 

\begin{lemma}\label{restictN} $\mu$ is an eigenvalue of $\sL_n$ iff $\bar{\mu}$ is an eigenvalue of $\sL_{-n}$. In particular, $\sL_{n}$ and $\sL_{-n}$, $n\in \Z\setminus\{0\}$, share the same real parts of eigenvalues.
\end{lemma} 

\begin{remark} The analogy of this proposition also works for $\sM_n$, $n\in \Z\setminus\{0\}$. i.e., $\mu$ is an eigenvalue of $\sM_n$ iff $\bar{\mu}$ is an eigenvalue of $\sM_{-n}$.
\end{remark}

\begin{proof}
Recall the decomposition $\sL_n = (A_n+B_n)A_n + C_n$ in Lemma \ref{invariance.sL}. The operators $A_n$, $B_n$ and $C_n$ defined by 
\eqref{op.ABCn} satisfy
\[
\overline{A_n h } = A_{-n} \bar h,\quad
\overline{B_n h } = B_{-n} \bar h,\quad
\overline{C_n h } = C_{-n} \bar h.
\]
This shows
\EQ{
\overline{\sL_n h} = \sL_{-n} \bar{h}.
}
This %
equality implies that 
\EQN{
\sL_n h = \mu h \iff \sL_{-n}\bar{h} = \bar{\mu}\bar{h},
}
and the statement of the lemma follows.
\end{proof}

By Lemma \ref{restictN}, it is enough to consider the eigenvalue problem for the linear operator $\sL_n$ only for $n\in \N$. The next lemma shows that it suffices to consider $n=1$.

\begin{lemma}\label{restictn=1} 
An eigenpair $(\mu,h)$ of $\sL_n$ for parameters $(a,\si)$ is also an eigenpair of $\sL_1$ for parameters $(a,n\si)$.
\end{lemma} 

It is because $n$ in the expression of $A_n$, $B_n$, $C_n$ and $\sL_n$ is always together with $\si$.
Therefore, it is enough to consider the eigenvalues of $\sL_1$ for any $a>1$ and $\si>0$.
 
 Our numerical evidences suggest the following.
 
\begin{conjecture}\label{conj.sL1} For all $a>1$ and $\si>0$, all eigenvalues of the linear operator $\sL_1$ have positive real-parts.
\end{conjecture}

We provide some numerical evidences for this conjecture. 

\medskip
\textbf{Step 1}.\ Finding the eigenvalues of the linear operator $\sL_1$.

As we did for $\sL_0$, by the finite difference scheme, the eigenvalue problem for $\sL_1$ can be written as the form of a matrix equation: 
\EQN{\label{ex.sL}
\begin{pmatrix}
I & -A_1\\
A_1 +B_1 & C_1
\end{pmatrix}
Y
=\mu
\begin{pmatrix}
0 & 0\\
0 & 1
\end{pmatrix}
Y.
}

\begin{table}[H]
\begin{center}
\caption{Minimum of real-part of eigenvalues of $\sL_1$, $N=640$}\label{min.sL1}
\begin{tabular}{|l||*{8}{c|}} 
\hline
\diagbox[width=12mm] {$\si$}{$a$} & 1.001 & 1.01 & 1.1 & 1.2 & 1.5 & 2 & 5 & 10\\
\hline
.001 & -5931.5 & -2.8719 &-0.0101 & -0.0025 & -0.0005 & -0.00015 & -0.00001 & 0.000001\\
\hline
0.01 & -0.0030 & 11.9444 & 0.1391 & 0.0209 & 0.0025 & 0.00091 & 0.00053 & 0.00051\\
\hline
0.1 & 12.0776 & 11.9809 & 11.7450 & 2.5152 & 0.2995& 0.1077 & 0.0546 & 0.0511\\
\hline
1 & 21.9890 & 22.0544 & 21.6837 & 20.8485 & 17.5098 & 10.5236 & 6.3884 & 6.0898\\
\hline
10 & 10911 & 10913 & 10975 & 11013 & 11048 & 11024 & 10706 & 10547\\
\hline
50 & 6272600 & 6273900 & 6281400 &6285600 &6291300 & 6293900 & 6288200 & 6279700\\
\hline
\end{tabular}
\end{center}
\end{table}

Then, by the assistance of MATLAB, we obtain Table \ref{min.sL1} %
of the minimum of real part of eigenvalues of $\sL_1$. %
From the values on the table, we observe that the minimum of the real part of the eigenvalues are positive except for when $\si\ll 1$. Thus, this supports the Conjecture \ref{conj.sL1} for $\si\gtrsim 1$. Also, as $\si$ increases, the minimum of real part of eigenvalues also increases. %

When $\si\ll 1$, Table \ref{min.sL1} suggests that we may have eigenvalues with negative real parts. However,
there's a possibility that negative approximated eigenvalues are obtained because of the approximation errors. 
This is supported by the following comparisons with the case $n=0$.

Recall that
\EQ{
\sL_1 h
&=(A_1+B_1) A_1h +C_1h\\
&=(A_0 + B_0 + \si^2 I)(A_0h + \si^2h) +\si^2 \left(1+\frac {2(a^2-1) }{(a-\cos \phi)^2}\right)h+ C_0 h\\
&\quad +i\si\left(\frac {2(a^2-1) }{(a-\cos \phi)^2}(A_0h + \si^2h)-B_0h \right)
+ 12 i\si \frac{(a^2-1)\sin^2\phi}{(a-\cos \phi)^4}h.\\
&= [(A_0+B_0)A_0 +C_0]h+Th = \sL_0 h+Th
}
where
\EQ{
Th 
=& \ i\si\left(\frac {2(a^2-1) }{(a-\cos \phi)^2} A_0h- B_0h+ \frac{12(a^2-1)\sin^2\phi}{(a-\cos \phi)^4}h\right)\\
&+\si^2\left(2 A_0h + B_0h+ h+\frac {2(a^2-1) }{(a-\cos \phi)^2}h\right)   
+i\si^3\frac {2(a^2-1) }{(a-\cos \phi)^2}h+ \si^4h
\\
=:& \ \si T_1h + \si^2T_2h+\si^3T_3h+\si^4T_4h. 
}

Since $T=O(\si)$, it can be considered as a perturbation to $\sL_0$ for sufficiently small $\si$. In other words, $\sL_1=\sL_0+T(\si)$ is a perturbed operator from $\sL_0$ when $\si\ll 1$. By the perturbation theory, we expect that the eigenvalues of $\sL_1$ are perturbations of those of $\sL_0$. This is evidenced by Table \ref{comparison}, where the operator $\sL_1$ numerically have a negative minimum of real part of the eigenvalues, for sufficiently small $\si$, (especially $\si=0.001$), because $\sL_0$ has a negative numerical minimum of real part of the eigenvalues.

\begin{table}[H]
\caption{Comparison between the minimums of real parts of eigenvalues. The notation 5.9662e+03 means $5.9662\cdot 10^{+03}$.
}\label{comparison}
\begin{center}
\begin{tabular}{|p{1.8cm}||*{7}{c|}} 
\hline
 &\diagbox{$N$}{$a$}& 1.001 & 1.01 & 1.1 & 1.2 & 2\\
\hline
\multirow{2}{*}{$n=0$} 
& 640 & -5.9662e+03 & -4.0271 & -0.0116 & -0.0027 & -1.6395e-04\\
& 900 & -0.24404e+03 & -1.9419 & -0.0059 & -0.0014 & -8.2981e-05\\
\hline
\multirow{2}{\linewidth}{$n=1$ $\si=0.001$}
& 640 &-5.9315e+03 & -2.8719 &-0.0101 & -0.0025 & -0.00015\\
& 900 & -2.3882e+03 & -0.5854 & -0.0044 & -0.001 & -7.2192e-05\\
\hline
\end{tabular}
\end{center}
\end{table}

\medskip
\textbf{Step 2}.\ Asymptotic analysis of the minimum of the real part of the eigenvalues of $\sL_1$ for small $\si$.

\smallskip

To resolve the issue for $\si \ll 1$ mentioned in the end of the previous step, we 
revise our numerical scheme based on asymptotic analysis. Since we already know that $(0,\pa_a\Psi)$ is an eigen-pair of $\sL_0$, we decompose an eigenfunction $h$ of $\sL_1$ as 
\EQ{
h = \pa_a\Psi + \eta, \quad \int_0^{\pi} \pa_a\Psi\cdot \eta\, d\phi =0. 
}
One may add a weight like $\sin \phi$ in the orthogonality condition. We skip it for simplicity.

Noting $\sL_0 \pa_a\Psi = 0$, we have the equation for the perturbation $\eta$: 
\begin{equation}\label{eta.eq}
\begin{cases}
\sL_0 \eta + T(\pa_a\Psi + \eta) = \mu (\pa_a\Psi + \eta)\\
\int_0^{\pi} \pa_a\Psi\cdot \eta\, d\phi =0\\
A_0\eta|_{\phi=0,\pi} = \eta|_{\phi=0,\pi} =0.
\end{cases}
\end{equation}

By matching the order of each term in $\si \ll 1$, we expect the expansion of $\mu$ and $\eta$ as
\EQN{
\mu = \sum_{k=1}^{\infty} \mu_k \si^k,\quad
\eta = \sum_{k=1}^{\infty} \eta_k \si^k,
}
and the equation of order $\si^1$ from \eqref{eta.eq} is
\EQN{\label{Osi-eq}
\sL_0 \eta_1 
&= -T_1 \pa_a\Psi+ \mu_1 \pa_a\Psi\\
&=i\left(-\frac {2(a^2-1) }{(a-\cos \phi)^2}A_0\pa_a\Psi 
+ B_0\pa_a\Psi 
-   \frac{12(a^2-1)\sin^2\phi}{(a-\cos \phi)^4}\pa_a\Psi\right)
+ \mu_1 \pa_a\Psi.
}

Since $\Re(\sL_0\eta_1) = \sL_0(\Re\eta_1)$, $\pa_a\Psi$ is a real-valued function, and the operators $A_0$ and $B_0$ map real-valued functions to real-valued, the real-part of \eqref{Osi-eq} is
\EQN{\label{eig.eta1}\begin{cases}
\sL_0 \Re(\eta_1) 
=\Re(\mu_1) \pa_a\Psi\\
\int_0^{\pi} \pa_a\Psi \cdot \Re(\eta_1) d\phi = 0\\
A_0\Re(\eta_1)|_{\phi=0,\pi}=\Re(\eta_1)|_{\phi=0,\pi} = 0.
\end{cases}}

By Theorem \ref{simple.eigenvalue}, the solution of \eqref{eig.eta1} should be 
\begin{align}\label{sol.first.order}
(\Re \mu_1, \Re \eta_1)=(0,0).
\end{align}
Indeed, when $\Re(\mu_1)= 0$, the solutions $\Re (\eta_1)$ of the first equation with the boundary conditions are constant multiples of $\pa_a\Psi$. Then, by the orthogonality condition, we obtain $\Re \eta_1=0$. On the other hand, when $\Re(\mu_1)\neq 0$, we have no solution because of the non-existence of strictly generalized eigenfunction for zero eigenvalue. We also observed \eqref{sol.first.order} numerically by solving
\EQN{\label{matrix.sL1.asymp}
\begin{pmatrix}
I & -A_0 & 0\\
A_0+B_0 & C_0 & -\pa_a\Psi\\
0 & (\pa_a\Psi)^T & 0
\end{pmatrix}
Y_1
=\begin{pmatrix}
0\\ 0 \\0
\end{pmatrix},
}
where $Y_1\in \R^{2N+1}$ is the discretization of $\Re (A_0 \eta_1, \eta_1, \mu_1)^T$,
\[
Y_1 = \bke{
A_0\Re(\eta_1)(\phi_1), \ldots, 
A_0\Re(\eta_1)(\phi_N),
\Re(\eta_1)(\phi_1) ,
\ldots,
\Re(\eta_1)(\phi_N) ,
\\ \Re(\mu_1)
}^T,
\]
and for $\pa_a\Psi$ in \eqref{matrix.sL1.asymp}, we use
\EQ{
\pa_a\Psi = (\pa_a\Psi(\phi_1),\ldots,\pa_a\Psi(\phi_N))^T.
}
The system \eqref{matrix.sL1.asymp} consists of $2N+1$ equations.
The first $N$ equations make $Y_1$ of the form $Y_1=(A_0 \xi, \xi, \nu)^T$.
The next $N$ equations correspond to the first equation in \eqref{eig.eta1}. The last equation in \eqref{matrix.sL1.asymp} corresponds to the orthogonality condition in \eqref{eig.eta1}. 
We omit our numerical results for \eqref{sol.first.order} as we have given a proof of it.

\medskip

Now, consider the real part of the equation of order $\si^2$ from \eqref{eta.eq}:
\EQN{\label{2nd.ord.eq}
\sL_0 \Re(\eta_2) -\Re(\mu_2)\pa_a\Psi
=& \ -T_2\pa_a\Psi+\Im(T_1)\Im(\eta_1)-\Im(\mu_1)\Im(\eta_1)\\
=& \  -(2A_0 + B_0)\pa_a\Psi - \left(1  + \frac {2(a^2-1) }{(a-\cos \phi)^2}\right)\pa_a\Psi \\
&+ \left(   \frac {2(a^2-1) }{(a-\cos \phi)^2} A_0- B_0 +12\frac{(a^2-1)\sin^2\phi}{(a-\cos \phi)^4}I \right)\Im(\eta_1)\\
&- \Im(\mu_1)\Im(\eta_1). 
}
Here, $\Re(\eta_2)$ and $\Re(\mu_2)$ are unknown and ($\Im(\mu_1)$, $\Im(\eta_1)$) can be obtained by solving the imaginary part of the $O(\si)$-equation \eqref{Osi-eq}. Also, we compute
\EQ{
(2A_0+B_0)\pa_a\Psi
&= -\frac{4 (a^2-1) \sin\phi}{(a-\cos\phi)^4}
=\frac{2 (a^2-1) }{(a-\cos\phi)^2}\pa_a\Psi.
}
In the similar way of solving the real part of the $O(\si)$-equation, we solve (\ref{2nd.ord.eq}) under the following boundary and orthogonality conditions:
\EQ{
A_0\Re(\eta_2)|_{\phi=0,\pi}=\Re(\eta_2)|_{\phi=0,\pi} =0, \quad 
\int_0^{\pi} \pa_a\Psi\cdot \Re(\eta_2) d\phi =0.
}

\begin{table}[H]
\caption{Values of $\Re(\mu_2)$}\label{mu2}
\begin{center}
\begin{tabular}{|l|*{7}{c|}}
\hline
\diagbox{$N$}{$a$}& 1.001 & 1.01& 1.1& 1.2& 2 & 10 & 100\\
\hline
320 & 40.4784& 13.2605& 6.8694 &6.0795 &5.2064&5.0067&5.0001\\
\hline
640  & 34.7380& 13.1886& 6.8677 &6.0790 &5.2063&5.0067&5.0001\\
\hline
1000 & 33.9805& 13.1748& 6.8674 &6.0788 &5.2063&5.0067&5.0001\\
\hline
2000 & 33.6191& 13.1677& 6.8673 &6.0788 &5.2063&5.0067&5.0001\\
\hline
3000 & 33.5545& 13.1664& 6.8672 &6.0788 &5.2063&5.0067&5.0001\\
\hline
\end{tabular}
\end{center}
\end{table}

Then, we obtain Table \ref{mu2} for the values of $\Re(\mu_2)$. 
Note that the eigenvalue $\mu$ of the linear operator $\sL_1$ satisfies
\EQN{\label{mu}
\Re(\mu) \sim \Re(\mu_2) \si^2
}
because of $\Re(\mu_1)=0$. Thus, the positiveness of the values of $\Re(\mu_2)$ implies that $\Re(\mu)$ is positive for sufficiently small $\si$. This provides the numerical evidence of the desired spectral property of $\sL_n$, $n\neq 0$, even in the case of $\si\ll 1$. 

On the other hand, on the Table \ref{mu2}, we observe that as $a$ goes to infinity, the value of $\Re(\mu_2)$ become stabilized.  Also, for sufficiently small $a$, as $N$ goes to infinity, $\Re(\mu_2)$ is relatively stable. Considering (\ref{mu}), the value of $\Re(\mu)$ is stable numerically for sufficiently small $\si$.

\subsection{Summary}
In this section, we analyzed the kernel and eigenvalues of $\sL_n$ both analytically and numerically, with the help of asymptotic analysis. Here is a summary for all $a>1$ and $\si>0$:
\begin{enumerate}
\item We proved that the kernel of $\sL_0$ is spanned by $\pd_a \Psi$. We also proved that $\sL_0$ is not symmetric with respect to any weight.
\item We presented numerical evidence that $\sL_0$ only has real eigenvalues, and the second smallest eigenvalue is positive. 
\item For $\sL_n$, $n\not =0$, we presented numerical evidence that $\sL_n$ has complex eigenvalues, 
and the real parts of all its eigenvalues are positive. In particular, it has no zero eigenvalue nor purely imaginary eigenvalue.
\item Our numerical observations support that there is no bifurcation  for $ 1.01<a<\infty$.
Conjectures \ref{conj.sL0} and \ref{conj.sL1} suggest that Landau solutions are stable under DSS no-swirl axisymmetric perturbations.
\end{enumerate}

\section*{Acknowledgments}
The research of both Kwon and Tsai was both partially supported by NSERC grant RGPIN-2018-04137 (Canada).

\bibliographystyle{abbrv}

Hyunju Kwon
\href{https://orcid.org/0000-0003-4093-3991}{\includegraphics[width=8pt]{ORCID-iD_icon-128x128.png}},
School of Mathematics, Institute for Advanced Study, Princeton, NJ 08540, USA.
E-mail: hkwon@ias.edu

\medskip

Tai-Peng Tsai
\href{https://orcid.org/0000-0002-9008-1136}{\includegraphics[width=8pt]{ORCID-iD_icon-128x128.png}}, 
Department of Mathematics, University of British Columbia, 
Vancouver, BC V6T 1Z2, Canada. E-mail: ttsai@math.ubc.ca

\end{document}